\documentclass{article}
\usepackage{amsmath,amsfonts,amssymb,amsthm}
\usepackage{color}

\usepackage{hyperref}

\providecommand{\R}{\mathbb{R}}
\providecommand{\C}{\mathbb{C}}
\providecommand{\N}{\mathbb{N}}

\providecommand{\eps}{\varepsilon}
\providecommand{\om}{\omega}

\def\longrightharpoonup{\relbar\joinrel\rightharpoonup} 
\def\cv{\stackrel{w}{\longrightharpoonup}}
\def\cvwstar{\stackrel{w*}{\longrightharpoonup}}

\renewcommand{\leq}{\leqslant}
\renewcommand{\geq}{\geqslant}
\renewcommand{\Re}{\mbox{Re}}
\renewcommand{\Im}{\mbox{Im}}

\renewcommand{\div}{\operatorname{div}}
\newcommand{\curl}{\operatorname{curl}}

\newcommand{\Id}{\operatorname{Id}}
\newtheorem{Theorem}{Theorem}

\newtheorem{Proposition}{Proposition}
\newtheorem{Lemma}{Lemma}
\newtheorem{Remark}{Remark}

\begin{document}

\author{Olivier Glass\footnote{Ceremade,
Universit\'e Paris-Dauphine, 
Place du Mar\'echal de Lattre de Tassigny,
75775 Paris Cedex 16, France
}, 
Christophe Lacave\footnote{Institut Math\'ematiques de Jussieu, 
Universit\'e Paris-Diderot - Paris 7,
175, rue du Chevaleret,
75013 Paris,
France
},
Franck Sueur\footnote{Laboratoire Jacques-Louis Lions,
Universit\'e Pierre et Marie Curie - Paris 6,
4 place Jussieu,
75005 Paris, 
France}
}
\date{\today}
\title{On the motion of a small body  immersed in a two dimensional incompressible perfect fluid.  }
\maketitle

\begin{abstract}
In this paper we prove that the motion of a solid body in a two dimensional  incompressible perfect fluid converges, when the body shrinks to a point with  fixed mass and  circulation, 
to a variant of the vortex-wave system where the vortex, placed in the point occupied by the shrunk body, is accelerated by a lift force similar to the Kutta-Joukowski force of the irrotational theory.
\end{abstract}

\section{Introduction}
\label{Intro}
In this paper we consider the motion of a small solid body in a planar ideal fluid, and the limit behaviour of the system as the solid body is reduced to a point. \par
Let us first describe the equations when the solid has a fixed size.
Let $\mathcal{S}_0$ be a closed, bounded, connected and simply connected subset of the plane with smooth boundary.
We assume that the body initially occupies the domain $\mathcal{S}_0$ and rigidly moves so that at  time $t$  it occupies an isometric  domain denoted by $\mathcal{S}(t)$.
We denote $\mathcal{F} (t) := \R^2  \setminus \mathcal{S}(t) $ the domain occupied by the fluid  at  time $t$ starting from the initial domain $\mathcal{F}_{0}  := \R^2 \setminus {\mathcal{S}}_{0} $. \par
The equations modelling the dynamics of the system then read
\begin{eqnarray}
\displaystyle \frac{\partial u }{\partial t}+(u  \cdot\nabla)u   + \nabla p =0 && \text{for} \ x\in \mathcal{F} (t), \label{Euler1}\\
\div u   = 0 && \text{for} \ x\in \mathcal{F}(t) , \label{Euler2} \\
u  \cdot n =   u_\mathcal{S} \cdot n && \text{for}  \  x\in \partial \mathcal{S}  (t),   \label{Euler3} \\
\lim_{|x|\to \infty} |u| =0,& & \\
m h'' (t) &=&  \int_{\partial  \mathcal{S} (t)} p n \, ds ,  \label{Solide1} \\ 
\mathcal{J} r' (t) &= &   \int_{\partial  \mathcal{S} (t)} (x-  h (t) )^\perp \cdot pn   \, ds , \label{Solide2} \\
u |_{t= 0} = u_0 & &  \text{for}  \  x\in  \mathcal{F}_0 ,  \label{Eulerci2} \\
h (0)= h_0 , \ h' (0)=  \ell_0 , & &   r  (0)=  r _0.  \label{Solideci}
\end{eqnarray}
Here $u=(u_1,u_2)$ and $p$ denote the velocity and pressure fields,
$m$ and $ \mathcal{J}$ denote respectively the mass and the moment of inertia of the body  while the fluid  is supposed to be  homogeneous of density $1$, to simplify the notations.
When $x=(x_1,x_2)$ the notation $x^\perp $ stands for $x^\perp =( -x_2 , x_1 )$, 
$n$ denotes  the unit normal vector pointing outside the fluid,  $h'(t)$
is the velocity of the center of mass  $h (t)$ of the body and $r(t)$ denotes the angular velocity of the rigid body. Finally we denote by $u_{{\mathcal S}}$ the velocity of the body:
\begin{equation} \label{VeloBody}
u_\mathcal{S} (t,x) =   h' (t)+ r (t) (x-  h (t))^\perp .
\end{equation}
\ \par
Since  $\mathcal{S} (t)$ is the position occupied by a rigid body there exists a rotation matrix 
\begin{eqnarray} \label{RotationMatrix}
Q (t):= 
\begin{bmatrix}
\cos  \theta (t) & - \sin \theta (t)
\\  \sin  \theta (t) & \cos  \theta (t)
\end{bmatrix},
\end{eqnarray}
such that  the position $\eta (t,x) \in \mathcal{S} (t)$  at 
the time $t$ of the point fixed to the body with an initial position $x$ is 
\begin{eqnarray} \label{FlotSolide}
\eta (t,x) := h (t) + Q (t)(x- h_0) .
\end{eqnarray}
The angle $\theta$ satisfies
\begin{equation*}
\theta'(t) = r (t),
\end{equation*}
and we choose $\theta (t)$ such that $\theta (0) =  0$.
\par
\ \par
The equations \eqref{Euler1} and \eqref{Euler2} are the incompressible Euler equations, the condition \eqref{Euler3} means that the boundary is impermeable and the equations \eqref{Solide1} and \eqref{Solide2} are the Newton's balance law for linear and angular momenta. \par 
\ \par
For the study of ideal flow, an important quantity is the vorticity $w:= \curl u=\partial_1 u_2- \partial_2 u_1$, satisfying the transport equation: 
\begin{equation} \label{Vort1}
\displaystyle \frac{\partial w }{\partial t}+(u  \cdot\nabla)w   =0 \  \text{for} \ x\in \mathcal{F} (t). 
\end{equation}
%
One has the following result concerning the Cauchy problem for the above system, the initial position of the solid being given. This result describes weak solutions, extending results concerning the fluid alone. It considers the case when vorticity belongs in $L^{p}$ as in DiPerna and Majda \cite{DiPernaMajda} and includes weak solutions with bounded vorticity as in Yudovich \cite{Yudovich}.
\begin{Theorem} \label{ThmYudo}
Let $p \in (2,+\infty]$.
For any $u_0 \in C^{0}(\overline{\mathcal{F}_0};\R^{2})$, $(\ell_0,r_0) \in \R^2 \times \R$, such that:
\begin{equation} \label{CondCompatibilite}
\div u_0 =0 \text{ in } {\mathcal F}_0 \ \text{ and } \  u_0   \cdot  n = (\ell_0 + r_0 (x-h_{0})^{\perp})   \cdot  n \text{ on } \partial \mathcal{S}_0,
\end{equation}
\begin{equation} \label{TourbillonYudo}
w_0 := \curl u_0  \in L_c^{p}(\overline{\mathcal F}_0),
\end{equation}
\begin{equation*}
\lim_{|x| \rightarrow +\infty} u_{0}(x) =0,
\end{equation*}
there exists a solution $(h',r,u)$ of \eqref{Euler1}--\eqref{Solideci} in $C^1 (\R^+; \R^2 \times \R) \times C^{0}(\R^+,  W^{1,p}({\mathcal F}(t)))$ with $\partial_{t} u, \nabla p \in L_{loc}^{\infty}(\R^{+},L^{q}({\mathcal F}(t)))$ for any $q \in (1,p ]$ when $p<\infty$ and in $C^1 (\R^+; \R^2 \times \R) \times L^{\infty}_{loc}(\R^+, \mathcal{LL}({\mathcal F}(t)))$ with $\partial_{t} u, \nabla p \in L^{\infty}_{loc}(\R^{+},L^{q}({\mathcal F}(t)))$ for any $q \in (1,+\infty)$ when $p=\infty$. \par
Moreover such a solution satisfies that  for all $t>0$, $w(t):=\curl(u(t)) \in L^p_c(\overline{{\mathcal F}(t)})$, it is energy-conserving in the sense of Proposition \ref{PropKirchoffSueur} and $\| w(t,\cdot)\|_{L^{q}({\mathcal F}(t))}$ (for any $q \in [1,p]$), $\int_{{\mathcal F}(t)} w(t,x) \, dx$ and $\int_{\partial {\mathcal S}(t)} u \cdot \tau \, ds$ are preserved over time. \par
Finally when $p= \infty$, the solution is unique.
\end{Theorem}
This result is proven in \cite{GS}. For the sake of self-containedness, we give a short proof of it in appendix. The notation $\mathcal{LL}(\Omega)$ refers to the space of log-Lipschitz functions on $\Omega$, that is the set of functions $f \in L^{\infty}(\Omega)$ such that
\begin{equation} \label{DefLL}
\| f \|_{\mathcal{LL}(\Omega)} := \| f\|_{L^{\infty}(\Omega)} + \sup_{x\not = y} \frac{|f(x)-f(y)|}{|(x-y)(1+ \ln^{-}|x-y|)|} < +\infty.
\end{equation}
Above, we used the abuse of notation $L^{\infty}(\R^{+};X({\mathcal F}(t)))$ (resp. $C^{0}(\R^{+};X({\mathcal F}(t)))$) where $X$ is a functional space; by this we refer to functions defined for almost each $t$ as a function in the space $X({\mathcal F}(t))$, and which can be extended as a function in $L^{\infty}(\R^{+};X(\R^{2}))$ (resp. $C^{0}(\R^{+};X(\R^{2}))$). \par
Let us also mention that the existence and uniqueness of finite energy classical solutions to the problem \eqref{Euler1}--\eqref{Solideci} has been tackled by Ortega, Rosier and Takahashi in \cite{ort1}. \par
\ \par
Let us now discuss the main problem considered in this paper, that is the behaviour of this system for a small body. Accordingly, we consider ${\mathcal S}_{0}$ a fixed domain as above, and define for $\varepsilon>0$ the resized domain $\mathcal{S}^{\eps}_0$ given by:
\begin{equation*}
\mathcal{S}^{\eps}_0 - h_{0} = \eps (\mathcal{S}_0 -h_{0}).
\end{equation*}
Therefore $\mathcal{S}^{\eps}_{0}$ denotes the domain initially occupied by the solid and 
\begin{equation*}
\mathcal{F}^{\eps}_{0}  := \R^2 \setminus {\mathcal{S}}^{\eps}_{0},
\end{equation*}
the one occupied by the fluid.
Let $w_0  \in L_c^{p}(\R^{2})$. We fix $\gamma,\ r_0\in \R$, and $h_0,\ \ell_0\in \R^2$ independently of $\eps$. Therefore, as we will see in Proposition \ref{PourFairePlaisirAFranckie}, there exists an unique vector field $u_0^\eps \in C^{0}(\overline{{\mathcal F}_{0}^{\varepsilon}};\R^{2})$ such that
\begin{equation} \label{UDI}
\left\{ \begin{array}{l}
\div u^\eps_0=0,\ \curl u^\eps_0=w_0^{\varepsilon} \text{ in } {\mathcal F}_{0}^{\varepsilon}, \\
u^\eps_0\cdot n=(\ell_0+r_0(x-h_0)^\perp)\cdot n  \text{ on } \partial {\mathcal S}_{0}^{\varepsilon},  \\
\lim_{|x|\to \infty} |u^\eps_0(x)|=0, \ \int_{\partial \mathcal{S}_0^\eps} u^\eps_0\cdot \tau\, ds =\gamma,
\end{array} \right.
\end{equation}
where
\begin{equation} \label{Defw0eps}
w_{0}^{\varepsilon} := w_{0|{\mathcal F}_{0}^{\varepsilon}}.
\end{equation}
We will be interested in the limit of the system as $\varepsilon \rightarrow 0^{+}$ in the following particular regime:
\begin{equation*}
m_{\varepsilon}=m \text{ and } {\mathcal J}_{\varepsilon}=\varepsilon^{2} {\mathcal J}_{0},
\end{equation*}
where $m$ and ${\mathcal J}_{0}$ are fixed constant. This is obtained for instance for a homogeneous solid, with a constant mass as $\varepsilon \rightarrow 0^{+}$. \par
\ \par

The main goal of this paper is to prove the following theorem.
\begin{Theorem}[]
\label{MR}
Assume that $p \in (2,+\infty]$.
Let  be given $h_{0} \in \R^{2}$, $\gamma \in \R$, $(\ell_0 ,r_0 ) \in \R^3$,  $ w_0 $ in $L^p_c (\R^2 )$. Consider $T>0$. For any $\eps \in (0,1]$, we associate $u^\eps_0$ by \eqref{UDI}-\eqref{Defw0eps} and consider $(h^{\varepsilon},r^{\varepsilon},u^\eps)$ a solution of the system \eqref{Euler1}--\eqref{Solideci} given by Theorem  \ref{ThmYudo}. \par
Then, up to a subsequence, one has the following:
\begin{itemize}
\item  $h^\eps $ converges to $h$ weakly-$*$ in $W^{2,\infty} (0,T;\R^{2})$, 
\item $\varepsilon \theta^{\varepsilon}$ converges to $0$ weakly-$*$ in $W^{2,\infty} (0,T;\R)$,
\item $w^{\varepsilon}$ converges to $w$ in $C^{0} ([0,T]; L^{p}(\R^{2})-w)$ (resp. in $C^{0} ([0,T]; L^{\infty}(\R^{2})-w*)$ if $p=+\infty$),
\item $u^{\varepsilon}$ converges to $\displaystyle \tilde{u} + \frac{\gamma}{2\pi} \frac{(x-h(t))^{\perp}}{|x-h(t)|^{2}}$ in $C^{0} ([0,T]; L^{q}_{loc} (\R^{2} ))$ for $q<2$,
\item  one has
\begin{gather}
\label{EulerPoint} 
\frac{\partial w }{\partial t}+ \div
\bigg( \left[ \tilde{u}+ \frac{\gamma}{2\pi} \frac{(x-h(t))^{\perp}}{|x-h(t)|^{2}} \right] w \bigg) = 0
\ \text{ in } \ [0,T] \times \R^{2}, \\
\label{PointEuler}
m h''(t) = \gamma \Big(h'(t) - \tilde{u}(t,h(t))\Big)^\perp, \\
\label{EP0}
w |_{t= 0}=  w_0 ,\ h(0) = h_0, \ h' (0) = \ell_0, \\
\label{EPU}
\tilde{u}(t,x) =  \frac{1}{2 \pi} \int_{ \R^{2}} \frac{(x-y)^{\perp}}{|x-y|^{2}} w (t,y) \, dy.
\end{gather}
\end{itemize}
\end{Theorem}
\begin{Remark}
Above the convergence of $w^{\varepsilon}$ holds when $w^{\varepsilon}$ is extended by $0$ inside the solid. In the same way, the convergence of $u^{\varepsilon}$ holds when extending it for instance by $0$ (or by $\ell^{\varepsilon} + r^{\varepsilon} (x- h(t))^{\perp}$) inside the solid.
\end{Remark}
\begin{Remark}
Equation  \eqref{EulerPoint} and the $w$-part of the initial data given in \eqref{EP0} 
hold in the sense that  for any test function $\psi\in C^\infty_c([0,T)\times\R^2)$ we have 
\begin{equation} \label{EqSolFaibleIntro}
\int_0^\infty \int_{\R^2} \psi_t  w \, \, dx\, dt
+\int_0^\infty \int_{\R^2} \nabla_x \psi \cdot \Big( \tilde{u}+ \frac{\gamma}{2\pi} \frac{(x-h(t))^{\perp}}{|x-h(t)|^{2}} \Big) w  \, dx\, dt
+\int_{\R^2} \psi(0,x) w_0 (x) \, dx=0 .
\end{equation}
\end{Remark}
Equation (\ref{EulerPoint}) describes the evolution of the vorticity of the fluid: it is transported by a velocity obtained by the usual Biot-Savart law in the plane, but from a vorticity which is the sum of the fluid vorticity and of a point vortex placed at the (time-dependent) position $h(t)$ where the solid shrinks, with a strength equal to the circulation $\gamma$ around the body. \par
Equation (\ref{PointEuler}) means that the shrunk body is accelerated by a lift force similar to the Kutta-Joukowski lift of the irrotational theory: the shrunk body experiments a lift which is proportional to the circulation $\gamma$ around the body and to the difference between the solid velocity and the virtual fluid velocity  obtained by the Biot-Savart law in the plane from the fluid vorticity, up to a rotation of a $\pi/2$ angle. See for instance the textbooks of Childress \cite{Childress} or Marchioro and Pulvirenti \cite{MP} for a discussion of the Kutta-Joukowski force. See also Grotta-Ragazzo, Koiller and Oliva \cite{ragazzo}, where they consider a similar system of a point mass embedded in an irrotational fluid and driven by Kutta-Joukowski force. \par
Let us mention that the problem of the limit of the Euler system around a fixed shrinking obstacle, which is tightly connected to ours, has been studied by Iftimie, Lopes-Filho and Nussenzveig-Lopes in \cite{ift_lop_euler}. Another result connected to our study is given in Dashti and Robinson \cite{DashtiRobinson}, where the authors consider the limit of a shrinking ball of fixed density and without rotation in a viscous fluid (modelled by the Navier-Stokes equations). \par

\begin{Remark}
A challenging open problem is to extend the previous analysis to the case where the density of the body  is fixed as $\eps$ goes to zero so that the mass of the body  is vanishing as the body  is shrinking to a point.  Formally the equations (\ref{EulerPoint})--(\ref{EPU})  would reduce to the vortex-wave system (for which we refer to \cite{MP}), but such a limit is quite singular as the equation \eqref{PointEuler} degenerates into a first order equation as $m$ goes to $0$.
\end{Remark}
\begin{Remark}
Since we start with a conservative and reversible system it is expected that the equations \eqref{EulerPoint}--\eqref{EPU} should be also 
conservative and reversible. This is actually the case and we  will even see in Section \ref{Sec:Hamilton} that Equations \eqref{EulerPoint}--\eqref{EPU} can be seen as an Hamiltonian system with respect to following renormalized energy
\begin{equation} \label{HamiltonienLimite}
2 \mathcal{H} =  m |h'(t)|^2  -  \int_{\R^2 \times \R^2 }  G (x-y) w (t,x) w (t,y) \, dx \, dy  -  2  \gamma \int_{\R^2  } G (x-h(t))  w (t,x) \, dx ,
\end{equation}
where 
\begin{equation} \label{LeG}
G (x) :=  \frac{1}{2 \pi} \ln |x|.
\end{equation}
\end{Remark} 
\begin{Remark}
Following the lines of \cite[Section 5.3]{ift_lop_euler}, one can see that in the above limit equation \eqref{EulerPoint} and \eqref{EPU} can be rewritten in the following velocity form:
\begin{gather*}
\partial_t  \tilde{u} +  (\tilde{u} \cdot \nabla)  \tilde{u} + \gamma \div \big(  \tilde{u} \otimes H(x-h(t)) + H(x-h(t)) \otimes \tilde{u} \big) - \gamma \tilde{u}(t,h(t))^\perp \delta_{h(t)} = - \nabla p , \\
\div \tilde{u}=0 \ \text{ and } \ \tilde{u}_{|t=0} =  \frac{1}{2 \pi} \int_{ \R^{2}} \frac{(x-y)^{\perp}}{|x-y|^{2}} w_{0} (y) \, dy,
\end{gather*}
with 
\begin{equation*}
H(x):=\frac{1}{2\pi} \frac{x^{\perp}}{|x|^{2}}.
\end{equation*}
\end{Remark}

\ \par
The structure of the paper is as follows. In Section \ref{Sec:Velocity}, we give a representation of a velocity field satisfying \eqref{UDI}. In Section \ref{Sec:EBF}, we discuss a change of variables allowing to rephrase the system in a fixed domain. In Section \ref{Sec:APE}, we give a priori estimates on the system. Section \ref{PF} is the central part where we study precisely the effect of pressure on the body as $\varepsilon$ tends to $0^{+}$. In Section \ref{Passage} we prove Theorem \ref{MR} by establishing compactness and obtaining the limit equation. Section \ref{TR} is devoted to the proof of several technical lemmas. In Section \ref{PreuveYudo} we briefly prove Theorem \ref{ThmYudo}. Finally in Section \ref{Sec:Hamilton}, we prove that the limit system obtained in Theorem \ref{MR} has a Hamiltonian structure. \par
\section{Representation of the velocity in the body frame}
\label{Sec:Velocity}
Without loss of generality and for the rest of the paper, we assume from now on that 
\begin{equation*}
h_{0}= 0,
\end{equation*}
which means that the body  is centered at the origin at the initial time $t=0$. \par
\ \par
In this section, we study the elliptic $\div$/$\curl$ system which allows to pass from the vorticity to the velocity field, in the body frame. In the whole paper and in this section in particular, we will need some arguments of elementary complex analysis: for the rest of the paper, we identify $\C$ and $\R^{2}$ through 
\begin{equation*}
(x_{1},x_{2})= x_{1} + i x_{2}.
\end{equation*}
\subsection{Green's function and  Biot-Savart operator}
\label{Sec:GreensFunction}
We denote by $G^{\eps} (x,y)$ the Green's function of $\mathcal{F}^{\eps}_{0}$ with Dirichlet boundary conditions. 
We also introduce the function $K^{\eps} (x,y)=\nabla^\perp G^{\eps} (x,y)$ known as the kernel of the Biot-Savart operator  $K^{\eps} [\om]$ which therefore acts on $\om \in L^p_c (\overline{\mathcal{F}^{\eps}_{0}})$  through the formula 
\begin{equation*}
  K^{\eps}[\om](x)= \int_{\mathcal{F}^{\eps}_{0}}K^{\eps} (x,y) \om(y) \, dy.
\end{equation*}
The following is classical.
\begin{Proposition}\label{propdefK}
Let $p \in (2,+\infty)$ (resp. $p=+\infty$).
Let  $\om\in L^p_c (\overline{\mathcal{F}^{\eps}_{0}})$. Then $K^{\eps}[\om]$ is in the H\"older space $C^{1-2/p}(\mathcal{F}^{\eps}_{0})$ of bounded H\"older  continuous functions of order $1-2/p$ (resp. in $\mathcal{LL}(\mathcal{F}^{\eps}_{0})$), divergence-free, tangent to the boundary and such that $\curl K^{\eps}[\om]=\om$. 
Moreover, it satisfies
\begin{equation} \nonumber
K^{\eps}[\om](x) = {\mathcal O}\left( \frac{1}{|x|^{2}}\right) \ \text{ as } x \rightarrow \infty,
\end{equation}
and is consequently square-integrable, and its circulation around $\partial \mathcal{S}^{\eps}_0$
is given by
\begin{equation}\label{circK}
 \int_{\partial \mathcal{S}^{\eps}_0  }   K^{\eps} [ \omega] \cdot {\tau} \, ds = -  \int_{\mathcal{F}^{\eps}_{0}  }  \omega \, dx ,
\end{equation}
where $\tau$ is the tangent unit vector field on $\partial {\mathcal S}^{\varepsilon}_{0}$.
\end{Proposition}
\ \par
As we work in the exterior of a single solid, we can have an explicit formula for $K^{\varepsilon}$ in terms of a biholomorphism ${\mathcal F}_{0} \rightarrow \C \setminus \overline{B}(0,1)$. To that purpose, let us select the unique such biholomorphism ${\mathcal T}:{\mathcal F}_{0} \rightarrow \C \setminus \overline{B}(0,1)$ such that the following development holds for some $(\beta,\tilde{\beta}) \in \R_{*}^{+} \times \C$:
\begin{equation} \label{Eq:DevtT}
{\mathcal T}(z) = \beta z + \tilde{\beta} + \mathcal{O}\left(\frac{1}{z}\right) \ \text{ as } \ z \rightarrow + \infty.
\end{equation}
This is possible since $\mathcal{S}_0$ is a bounded, closed, connected and simply connected domain of the plane (using Riemann's mapping theorem and a conjugation by $z \mapsto 1/z$). Now, as $\mathcal{S}_0^\eps=\eps \mathcal{S}_0$, we can introduce ${\mathcal T}_\eps$ as the biholomorphism from $\mathcal{F}^\eps_0$ to the exterior of the unit ball given by: 
\begin{equation}\label{Teps}
{\mathcal T}_\eps(z)={\mathcal T}(z/\eps).
\end{equation}
In particular ${\mathcal T}={\mathcal T}_{1}$. \par
With these notations we have (see e.g. \cite{ift_lop_euler}):
\begin{equation} \label{Gepsilon}
G^{\varepsilon}(x,y)= \frac{1}{2\pi} 
\ln \frac{|{\mathcal T}_{\varepsilon}(x) - {\mathcal T}_{\varepsilon}(y) |}{|{\mathcal T}_{\varepsilon}(x) - {\mathcal T}_{\varepsilon}(y)^{*}| |{\mathcal T}_{\varepsilon}(y)|} ,
\end{equation}
and
\begin{equation*}
K^\eps[\omega](x)=\frac{1}{2\pi} D{\mathcal T}_\eps^T(x)\int_{\mathcal{F}^\eps_0}
\Bigl( \frac{{\mathcal T}_\eps(x)-{\mathcal T}_\eps(y)}{|{\mathcal T}_\eps(x)-{\mathcal T}_\eps(y)|^2}- \frac{{\mathcal T}_\eps(x)-{\mathcal T}_\eps(y)^*}{|{\mathcal T}_\eps(x)-{\mathcal T}_\eps(y)^*|^2}\Bigl)^\perp
\omega(y)\, dy,
\end{equation*}
with the notation 
\begin{equation*}
y^* =\frac{y}{|y|^2}.
\end{equation*}
These explicit formulas will help us to find estimates for the velocity in terms of vorticity estimates. \par
\ \par
Let us also introduce the Biot-Savart operator associated to the full plane, that is the operator, denoted $K_{\R^{2}}$ which maps a vorticity $\omega$ to the velocity
\begin{equation}
\label{BSR2}
K_{\R^{2}} \lbrack \omega \rbrack (x) :=  \int_{ \R^{2}} H (x-y)  \omega (y) \,dy , 
\end{equation}
where $H$ is defined as
\begin{equation} \label{defH}
H(x):= \frac{x^{\perp}}{2 \pi |x|^{2}}.
\end{equation}
Proposition \ref{propdefK} is valid on the whole plane (see e.g. \cite{CheminSMF}). Precisely we have
\begin{Proposition}\label{propdefKwhole}
Let $p \in (2,+\infty)$ (resp. $p=+\infty$).
There exists a constant $C>0$ such that the following holds true.
Let  $\om\in L^p_{c} (\R^2)$. Then $K_{\R^{2}}  [\om]$ is bounded, continuous, divergence-free and such that $\curl K_{\R^{2}} [\om]=\om$. Moreover, it satisfies
\begin{equation} \nonumber
\| K_{\R^{2}} [\om]  \|_{W^{1,p} (\R^2)}  + \| K_{\R^{2}} [\om]  \|_{C^{1-2/p} (\R^2)}  \leq  C (\| \om  \|_{L^p (\R^2)} +  \| \om  \|_{L^1 (\R^2)} ),
\end{equation}
\begin{equation*}
\text{(resp. }\| K_{\R^{2}} [\om]  \|_{\mathcal{LL} (\R^2)} \leq  C (\| \om  \|_{L^\infty (\R^2)} +  \| \om  \|_{L^1 (\R^2)})\text{)}, 
\end{equation*}
and
\begin{equation} \nonumber
K_{\R^{2}} [\om](x) = {\mathcal O}\left( \frac{1}{|x|}\right) \ \text{ as } x \rightarrow \infty.
\end{equation}
\end{Proposition}
We will also use several times the fact that $K_{\R^{2}}$ commutes with translations and rotations in the plane. \par
\subsection{Harmonic field}
To take the velocity circulation around the body into account, the following vector field will be useful. There exists one and only one solution  $H^{\eps}$ vanishing at infinity of 
\begin{gather*}
\div H^{\eps} = 0 \quad   \text{for}  \ x\in  \mathcal{F}^{\eps}_{0}, \\
\curl H^{\eps} = 0 \quad   \text{for}  \ x\in  \mathcal{F}^{\eps}_{0}, \\
H^{\eps} \cdot n = 0 \quad   \text{for}  \ x\in   \partial \mathcal{S}^{\eps}_0, \\
\int_{\partial \mathcal{S}^{\eps}_0 } H^{\eps} \cdot \tau \, ds = 1 .
\end{gather*}
We refer for instance to \cite{Kikuchi83}, \cite{ift_lop_euler}. This solution is smooth.
The vector field $H^{\eps}$ admits a harmonic stream function $ \Psi_{H^{\varepsilon}} (x)$:
\begin{equation*}
H^{\varepsilon} = \nabla^{\perp} \Psi_{H^{\varepsilon}},
\end{equation*}
which vanishes on the boundary $ \partial \mathcal{S}^{\eps}_0$, and behaves like $\ln |x |$ as $x$ goes to infinity. In our case, we have
\begin{equation} \label{DefPsiHEpsilon}
\Psi_{H^{\varepsilon}}(x)=\frac{1}{2\pi} \ln |{\mathcal T}_\eps(x)| \text{ and }
H^\eps(x) =\frac{1}{2\pi} D{\mathcal T}_\eps^T(x) \frac{({\mathcal T}_\eps(x))^\perp}{|{\mathcal T}_\eps(x)|^2}.
\end{equation}
The scaling law for $H^{\varepsilon}$ is as follows:
\begin{equation} \label{ScalingH}
H^{\varepsilon}(x) = \frac{1}{\varepsilon} H^{1} \left( \frac{x}{\varepsilon}\right).
\end{equation}
We develop the function $H^{\varepsilon}_{1}-iH^{\varepsilon}_{2}$ in Laurent series. The fact that it is holomorphic (as a function of $z=x_{1}+ix_{2}$), comes from $\curl H =\div H =0$ which translates into the Cauchy-Riemann equations. One can see that $(H^{\varepsilon}_{1}-iH^{\varepsilon}_{2})(z) = a^{\varepsilon}_{1}/z+ {\mathcal O}(1/z^{2})$ as $z \rightarrow \infty$ from the behaviour of $H^{\varepsilon}$ at infinity. To identify $a^{\varepsilon}_{1}$, we use the fact that
\begin{equation*}
\int_{\partial {\mathcal S}^{\varepsilon}_{0}} H^{\varepsilon}\cdot n \, ds =0 \ \text{ and } \ 
\int_{\partial {\mathcal S}^{\varepsilon}_{0}} H^{\varepsilon} \cdot \tau \, ds =1.
\end{equation*}
Hence we deduce that
\begin{equation} \label{HSeriesLaurent}
(H^{\varepsilon}_{1}-iH^{\varepsilon}_{2}) (z) = \frac{1}{2i\pi z} + {\mathcal O}(1/z^{2}) \ \text{ as } z \rightarrow \infty.
\end{equation}
Going back to real variables we have
\begin{equation} \label{HalInfini}
H^1 = {\mathcal O}\left(\frac{1}{|x|}\right) \text{ and } \nabla H^1 = {\mathcal O}\left(\frac{1}{|x|^{2}}\right).
\end{equation}
Let us also observe other consequences of \eqref{HSeriesLaurent}:
\begin{equation} \label{HSeriesLaurent3}
x^{\perp} \cdot H^{1} = \frac{1}{2 \pi} + {\mathcal O}\left(\frac{1}{|x|}\right) ,
\end{equation}
and
\begin{equation} \label{HSeriesLaurent2}
(H^1)^{\perp} - x^\perp  \cdot\nabla H^{1} = {\mathcal O}\left(\frac{1}{|x|^2}\right) .
\end{equation}
\begin{Remark} \label{casdisk}
In the case of a disk, we have 
\begin{eqnarray*}
H^{\eps} (x) = \frac{1}{2 \pi} \frac{x^{\perp}}{| x |^{2}}=H^1(x)=H(x).
\end{eqnarray*}
\end{Remark}
\subsection{Kirchoff potentials}
\label{KirPo}
Now to lift harmonically the boundary conditions, we will make use of  the  Kirchoff potentials, which are the solutions $\Phi^{\eps}:=(\Phi^{\eps}_{i})_{i=1,2, 3}$ of the following problems: 
\begin{equation}
\label{t1.3sec}
-\Delta \Phi^{\eps}_i = 0 \quad   \text{for}  \ x\in \mathcal{F}^{\eps}_{0}   ,
\end{equation}
\begin{equation}
\label{t1.4sec}
\Phi^{\eps}_i \longrightarrow 0 \quad  \text{for}  \ x \rightarrow  \infty, 
\end{equation}
\begin{equation}
\label{t1.5sec}
\frac{\partial \Phi^{\eps}_i}{\partial n}=K_i 
\quad  \text{for}  \  x\in \partial \mathcal{F}^{\eps}_{0}   ,
\end{equation}
where
\begin{equation} \label{t1.6sec}
(K_{1},\, K_{2}, \, K_{3}) :=(n_1,\, n_2 ,\, x^\perp \cdot n).
\end{equation}
Note that $K_{1}$, $K_{2}$ and $K_{3}$ actually depend on $\varepsilon$. 
Changing variables $y=x/\eps$, we see that
\begin{gather} \label{phi-scaling}
\Phi_{i}^\eps(x)= \eps \Phi_{i}^1(x/\eps) \ \text{ for } i=1,2, \\
 \label{phi-scaling2}
\Phi_{3}^\eps(x)= \eps^{2} \Phi_{3}^1(x/\eps).
\end{gather}
The existence of $\Phi_{i}^{1}$ is classical; note in particular that for $i=1,2,3$ one has
\begin{equation*}
\int_{\partial {\mathcal S}_{0}} K_{i} \, ds =0.
\end{equation*}
Now, $\Phi^{1}_{i}$ is the real part of a holomorphic function admitting a development in Laurent series; hence $\Phi_{i}^{1}(x) = {\mathcal O}(1/|x|)$ at infinity. For what concerns $\nabla \Phi^{1}_{i}$, we see that $\partial_{1} \Phi^{1}_{i} - i\partial_{2} \Phi^{1}_{i}$ is holomorphic and admits a development in Laurent series, which is the derivative with respect to $z$ of the former. We deduce that 
\begin{equation} \label{ComportementPhii}
\Phi_{i}^1(x) = {\mathcal O}\left(\frac{1}{|x|}\right) \text{ and }
\nabla \Phi_{i}^1(x) = {\mathcal O}\left(\frac{1}{|x|^{2}}\right) \text{ as } |x| \rightarrow +\infty. 
\end{equation}
and consequently that $\nabla \Phi^{\eps}_i$, for $i=1,2, 3$, are in $L^2 (\mathcal{F}^{\eps}_{0})$. 
%
%
%
%
%
%
\begin{Remark}
\label{casdisk2}
In the case of a disk, we have 
\begin{eqnarray*}
(\Phi^1_{1} , \Phi^1_{2}) = 2\pi (H^{1})^{\perp} , \quad \Phi_{3}=0.
\end{eqnarray*}
\end{Remark}
\subsection{Velocity decomposition}
Using the functions defined above, we deduce the following proposition.
\begin{Proposition}\label{PourFairePlaisirAFranckie}
Let $p>2$. Let be given $\om$ in $L^p_c (\mathcal{F}^{\eps}_{0})$, $\ell$ in $\R^2$, $r$ and $\gamma $ in  $\R$.
Then there is a unique solution $v$ in $W^{1,p}({\mathcal F}_{0}^{\varepsilon})$ when $p<+\infty$
(resp. in $\mathcal{LL}({\mathcal F}_{0}^{\varepsilon})$ when $p=+\infty$) of
\begin{equation} \label{DivCurlSystem}
\left\{ \begin{array}{l} 
 \div v = 0,  \quad   \text{for}  \ x\in  \mathcal{F}^{\eps}_{0} , \\
 \curl v  = \omega  \quad  \text{for}  \ x\in  \mathcal{F}^{\eps}_{0}  , \\
 v \cdot n = \left(\ell+r x^\perp\right)\cdot n  \quad   \text{for}  \ x\in \partial \mathcal{S}^{\eps}_0  , \\
 v \longrightarrow 0 \quad  \text{as}  \ x \rightarrow  \infty, \\
 \int_{ \partial \mathcal{S}^{\eps}_0} v  \cdot  \tau \, ds=  \gamma .
\end{array} \right.
\end{equation}
Moreover $v$ is given by 
\begin{equation}
\label{vdecomp}
v = K^{\eps} [\omega ] + (\gamma + \alpha )  H^{\eps} + \ell_1 \nabla \Phi^{\eps}_1 + \ell_2 \nabla \Phi^{\eps}_2
+ r \nabla \Phi^{\eps}_3 ,
\end{equation}
with 
\begin{equation} \label{DefAlpha}
\alpha :=   \int_{\mathcal{F}^{\eps}_{0}  }  \omega \, dx .
\end{equation}
\end{Proposition}
\begin{proof}[Proof of Proposition \ref{PourFairePlaisirAFranckie}]
The existence comes from the above paragraphs. The uniqueness can be easily deduced from \cite[Lemma 2.14]{Kikuchi83}.
\end{proof}
Let us also introduce the following variant of the Biot-Savart operator $K^\eps$, the so-called
hydrodynamic Biot-Savart operator $K^\eps_H $ which can here be deduced from $K^\eps$ by the formula 
\begin{equation*}
K^\eps_H = K^\eps +  \alpha H^{\varepsilon}.
\end{equation*}
As a consequence it satisfies 
\begin{eqnarray*}
\div K^\eps_H [\omega]= 0, \ \text{ for }  \ x\in  \mathcal{F}^{\eps}_{0}, \\
\curl K^\eps_H [\omega]  = \omega , \ \text{ for } \ \ x\in  \mathcal{F}^{\eps}_{0}, \\
K^\eps_H [\omega] \cdot n = 0, \ \text{ for } \ x\in   \partial \mathcal{S}^{\eps}_0, \\
\int_{  \partial \mathcal{S}^{\eps}_0} K^\eps_H [\omega]  \cdot  \tau\, ds  = 0 .
\end{eqnarray*}
Then $v$ can be decomposed as 
\begin{eqnarray}
\label{vitedechydro}
v =  K^{\eps}_H [\omega] + \gamma H^\eps + \ell_1 \nabla \Phi^\eps_1
+ \ell_2 \nabla \Phi^\eps_2 + r \nabla \Phi^{\varepsilon}_{3}  .
\end{eqnarray}
%
%
%
We also introduce the hydrodynamic Green function $G_{H}^{\varepsilon}$ as
\begin{eqnarray} \label{GepsHydro}
G^{\eps}_{H} (x,y) &:=& G^{\eps} (x,y) +  \Psi_{H^\eps} (x) +  \Psi_{H^\eps} (y) \\
\nonumber
&=& \frac{1}{2\pi} 
\ln \frac{|{\mathcal T}_{\varepsilon}(x) - {\mathcal T}_{\varepsilon}(y) | |{\mathcal T}_{\varepsilon}(x)|}{|{\mathcal T}_{\varepsilon}(x) - {\mathcal T}_{\varepsilon}(y)^{*}|} .
\end{eqnarray}
Consequently one has
\begin{equation*}
K_{H}^{\eps}[\om](x)= \int_{\mathcal{F}^{\eps}_{0}} \nabla^{\perp} G_{H}^{\eps} (x,y) \om(y) \, dy.
\end{equation*}
\section{Equations in the body frame}
\label{Sec:EBF}
\subsection{Velocity equation}
\label{Subsec:VEBF}
In order to transfer the equations in the body frame we apply the following isometric change of variable:
\begin{equation}  \label{chgtvar}
\left\{
\begin{array}{l}
 v^{\eps} (t,x)=Q^{\eps} (t)^T\  u^\eps(t,Q^{\eps}(t)x+h^{\eps}(t)), \\
 q^{\eps} (t,x)=p^\eps(t,Q^{\eps}(t)x+h^{\eps}(t)), \\
 {\ell}^{\eps} (t)=Q^{\eps} (t)^T \ (h^{\eps})' (t) .
\end{array}\right.
\end{equation}
so that the equations  (\ref{Euler1})-(\ref{Solideci})  become
\begin{eqnarray}
\label{Euler11}
\displaystyle \frac{\partial v^{\eps}}{\partial t}
+ \left[(v^{\eps}-\ell^{\eps}-r^{\eps} x^\perp)\cdot\nabla\right]v^{\eps} 
+ r^{\eps} (v^{\eps})^\perp+\nabla q^{\eps} =0 && x\in \mathcal{F}^{\eps}_{0} ,\\
\label{Euler12}
\div v^{\eps} = 0 && x\in \mathcal{F}^{\eps}_{0} , \\
\label{Euler13}
v^{\eps}\cdot n = \left(\ell^{\eps} +r^{\eps} x^\perp\right)\cdot n && x\in \partial \mathcal{S}^{\eps}_0, \\
\label{Solide11}
m (\ell^{\eps})'(t)=\int_{\partial \mathcal{S}_0^{\eps}} q^{\eps} n \ ds-mr^{\eps} (\ell^{\eps})^\perp & & \\
\label{Solide12}
\mathcal{J}_{\varepsilon} (r^{\eps})'(t)=\int_{\partial \mathcal{S}^{\eps}_0} x^\perp \cdot q^{\eps} n \ ds & &  \\
\label{Euler1ci}
v^{\eps}(0,x)= v^{\eps}_0 (x) && x\in \mathcal{F}^{\eps}_{0} ,\\
\label{Solide1ci}
\ell^{\eps}(0)= \ell_0,\ r^{\eps} (0)= r_0 . 
\end{eqnarray}
\subsection{Vorticity equation}
We define
\begin{eqnarray} \label{chgtvar2}
\omega^{\varepsilon}(t,x) &:=& w^{\varepsilon}(t, Q^{\eps}(t)x+h^{\eps}(t) ) \\
\nonumber
&=& \curl v^{\varepsilon}(t,x).
\end{eqnarray}
Taking the curl of the equation \eqref{Euler11} we get
\begin{equation} \label{vorty1}
\partial_t  \omega^{\eps} + \left[(v^{\eps}-\ell^{\eps}-r^{\eps} x^\perp)\cdot\nabla\right]  \omega^{\eps} =0 \text{ for }
x \in \mathcal{F}^{\eps}_{0} .
\end{equation}
Due to the conservations mentioned in Theorem \ref{ThmYudo}, we have
\begin{gather}
\nonumber
\gamma =  \int_{  \partial \mathcal{S}^{\eps}_0} v^{\eps}  \cdot  \tau \, ds = \int_{  \partial \mathcal{S}^{\eps}(t)} u^{\eps}  \cdot  \tau \, ds
= \int_{  \partial \mathcal{S}^{\eps}_{0}} u^{\eps}_{0}  \cdot  \tau \, ds, \\
\label{DefAlpha2}
\alpha^{\varepsilon} = \int_{ \mathcal{F}^{\eps}_0} \omega^{\eps}(t,x) \, dx = \int_{\mathcal{F}^{\eps}(t)} w^{\eps}(t,x) \, dx = \int_{\mathcal{F}^{\eps}_{0}} w^{\varepsilon}_{0}(x) \, dx.
\end{gather}

%
Now using Section \ref{Sec:Velocity}, we can recover the velocity of the fluid from the vorticity, the velocity of the rigid body and the circulation of the flow around the solid through the following formula:
\begin{equation}
\label{vitedec2}
v^{\eps} =  K_{H}^{\eps}[\omega^\eps]+ \gamma H^{\eps} + \ell^{\eps}_1 \nabla \Phi^\eps_1
+ \ell^{\eps}_2 \nabla \Phi^\eps_2 + r^{\varepsilon} \nabla \Phi_{3}^{\varepsilon}  .
\end{equation}
Also, introducing $\tilde{v}^{\varepsilon}$ by
\begin{equation}
\label{allo}
\tilde{v}^{\eps}: = v^{\eps}- \gamma H^{\eps},
\end{equation}
we have
\begin{equation}
\label{allo2}
\tilde{v}^{\eps} =  K_{H}^{\eps}[\omega^\eps] + \ell^{\eps}_1 \nabla \Phi^\eps_1
+ \ell^{\eps}_2 \nabla \Phi^\eps_2 + r^{\varepsilon} \nabla \Phi_{3}^{\varepsilon}.
\end{equation}
%
%

%
%
%
%
%
%
%
%
%
%
\section{A priori estimates}
\label{Sec:APE}
The goal of this section is to derive a priori bounds on the solutions given by Theorem \ref{ThmYudo}, independently of $\varepsilon$ (see Proposition \ref{Prop:APE} below). In particular, $p$ is fixed in $(2,+\infty]$. We also give a result on an approximation of the velocity as $\varepsilon \rightarrow 0^{+}$ (Proposition \ref{Pr}), which will be useful in the sequel. \par
\subsection{Vorticity}
Due to Theorem \ref{ThmYudo}, the generalized enstrophies are conserved when time proceeds, in particular, we have for any $t>0$,
\begin{eqnarray} \label{ConsOmega}
\| \omega^{\eps} (t, \cdot) \|_{L^{p}(\mathcal{F}^{\eps}_{0})} 
= \| w^{\varepsilon}_{0} \|_{ L^{p}( \mathcal{F}^{\eps}_{0} )}
\leq \| w_{0} \|_{ L^{p}( \R^{2})},
\quad \| \omega^{\eps}(t,\cdot) \|_{ L^{1}(\mathcal{F}^{\eps}_{0})} 
= \| w^{\varepsilon}_{0} \|_{ L^{1}( \mathcal{F}^{\eps}_{0} )}
\leq \| w_{0} \|_{ L^{1}(\R^{2})}    .
\end{eqnarray}
\subsection{Energy}
\label{NRJ}
Let us introduce the matrix 
\begin{equation} \label{InertieMatrix}
\mathcal{M}^{\eps} :=\mathcal{M}^{\varepsilon}_1+\mathcal{M}^{\eps}_2,
\end{equation}
where
\begin{equation} \label{EvoMatrice2sec}
\mathcal{M}^{\varepsilon}_1 := \begin{bmatrix} m \Id_2 & 0 \\ 0 & \mathcal{J}_{\varepsilon} \end{bmatrix}, \ 
\text{ and }
\mathcal{M}^{\eps}_2
:= \begin{bmatrix} \int_{\mathcal{F}^{\eps}_0} \nabla \Phi^{\eps}_a \cdot \nabla \Phi^{\eps}_b \ dx \end{bmatrix}_{a,b \in \{1,2,3\}}
=   \begin{bmatrix} \displaystyle {\eps}^{2 + \delta_{a,3} + \delta_{b,3}} \int_{{\mathcal F}_{0}} \nabla \Phi^{1}_{a} \cdot \nabla \Phi^{1}_{b} \end{bmatrix}_{a,b \in \{1,2,3\}}.
\end{equation}
The matrix $\mathcal{M}^{\eps}$ is symmetric and positive definite. The matrix $\mathcal{M}^{\eps}_2$ actually encodes the phenomenon of added mass, which, loosely speaking,  measures how much the  surrounding fluid resists the acceleration as the body moves through it. \par
Using this added mass matrix, one can deduce a conserved quantity.
\begin{Proposition} \label{PropKirchoffSueur}
The following quantity is conserved along the motion:
\begin{eqnarray*}
2 \mathcal{H}^{\eps} = X^T \mathcal{M}^{\varepsilon} X
- \int_{\mathcal{F}^{\varepsilon}_0 \times  \mathcal{F}^{\varepsilon}_0 }  G^{\varepsilon}_{H} (x,y) \omega^{\varepsilon} (x) \omega^{\varepsilon} (y) \, dx \, dy
-  2  \gamma \int_{\mathcal{F}^{\varepsilon}_0  } \omega^{\varepsilon} (x)  \Psi_{H^{\varepsilon}} (x) \, dx ,
\end{eqnarray*}
where 
\begin{equation*}
X := \begin{pmatrix} \ell^{\varepsilon}_1 \\ \ell^{\varepsilon}_2 \\ r^{\varepsilon} \end{pmatrix}.
\end{equation*}
%
%
%
\end{Proposition} 
The proof of Proposition \ref{PropKirchoffSueur} is given in Section \ref{TR}. \par
We will need the following technical lemma in order to derive some uniform estimates from Proposition \ref{PropKirchoffSueur}.
\begin{Lemma}\label{intln}
Let $f$ in $L^1 (\R^{2})\cap L^p (\R^2)$. We denote by 
\begin{equation} \label{DefRhof}
\rho_f :=\inf \, \{ d >1 \ / \  {\rm Supp}(f) \subset B(0,d) \}.
\end{equation}
Then there exists $C>0$ such that
\begin{equation*}
\int_{\R^2} \Bigl| \ln |x-y| f(x) \Bigl| \, dx \leq C\|f \|_{L^p} + \ln(2\rho_f) \|f\|_{L^1},
\end{equation*}
for any $y\in B(0,\rho_f)$.
\end{Lemma}
\begin{proof}
We fix $y\in B(0,\rho_f)$ and we decompose the integral:
\begin{eqnarray*}
\int_{\R^2}  \Bigl| \ln |x-y| f(x) \Bigl| \, dx &=&\int_{|x-y|\leq 1}  \Bigl| \ln |x-y| f(x) \Bigl| \, dx + \int_{|x-y|\geq 1}  \Bigl| \ln |x-y| f(x) \Bigl| \, dx \\
&\leq& \|f \|_{L^p}  \| \ln |\cdot|  \|_{L^{p'}(B(0,1))} + \ln(2\rho_f) \|f\|_{L^1},
\end{eqnarray*}
where
\begin{equation*}
p' := \frac{p}{p-1}.
\end{equation*}
This ends the proof.
\end{proof}
As a consequence we have the following result.
\begin{Proposition}\label{ProRata}
One has the following estimate for some constant $C=C ( m, {\mathcal J}_{0}, \| w_{0} \|_{L^{1} \cap L^{p}} , |\ell_{0}|, |r_{0}|,|\gamma|, \rho_{w_{0}})$, depending only on these values and the geometry for $\varepsilon=1$:
\begin{equation} \label{EstNRJ}
| \ell^{\varepsilon}(t) | + |\varepsilon r^{\varepsilon}(t)| \leq C [1 + \ln (\rho^{\varepsilon}(t))],
\end{equation}
where
\begin{equation*}
\rho^{\varepsilon}(t) := \rho_{\omega^{\varepsilon}(t,\cdot)} =\inf \{ d >1 \ / \  {\rm Supp}(\omega^{\varepsilon}(t,\cdot)) \subset B(0,d) \}.
\end{equation*}
\end{Proposition}
\begin{proof}
We first add a constant in time to ${\mathcal H}^{\varepsilon}$, in order to get a quantity which is bounded with respect to $\varepsilon$:
\begin{equation*}
\hat{\mathcal H}^{\varepsilon}:= {\mathcal H}^{\varepsilon} - \frac{1}{2}\ln(\varepsilon) (\alpha^{\varepsilon})^{2} - \ln(\varepsilon) \gamma\alpha^{\varepsilon},
\end{equation*}
where $\alpha^{\varepsilon}$ is given by \eqref{DefAlpha2}. 
We decompose
\begin{eqnarray*}
2\hat{\mathcal H}^{\varepsilon}&=& 
m |\ell^{\varepsilon} |^2 + \mathcal{J}_0 ({\varepsilon}r^{\varepsilon})^2 + (X^{\varepsilon})^T \mathcal{M}_2^{\varepsilon} X^{\varepsilon}  \\
&&-  \int_{\mathcal{F}^{\varepsilon}_0 \times  \mathcal{F}^{\varepsilon}_0 }
\Bigl(G^{\varepsilon}_{H} (x,y) + \ln(\varepsilon) \Bigl)\omega^{\varepsilon} (x) \omega^{\varepsilon} (y) \, dx \, dy
-  2  \gamma \int_{\mathcal{F}^{\varepsilon}_0  } \omega^{\varepsilon} (x) \Bigl( \Psi_{H^{\varepsilon}}  (x)+\ln(\varepsilon)\Bigl) \, dx,
\end{eqnarray*}
and we denote the last two terms as follows:
\begin{equation} \label{decompoH}
2\hat{\mathcal H}^{\varepsilon} =: m |\ell^{\varepsilon} |^2 + \mathcal{J}_0 ({\varepsilon}r^{\varepsilon})^2 
+ (X^{\varepsilon})^T \mathcal{M}_2^{\varepsilon} X^{\varepsilon} -R_1^{\varepsilon} -2\gamma R_2^{\varepsilon}.
\end{equation}
We begin by estimating $R_1^{\varepsilon}$ using \eqref{Gepsilon} and \eqref{GepsHydro}:
\begin{equation*}
R_1^{\varepsilon}=  
\frac{1}{2\pi}\int_{\mathcal{F}^{\varepsilon}_0 \times  \mathcal{F}^{\varepsilon}_0 } \Bigl(
\ln \Big| \varepsilon {\mathcal T} \left(\frac{x}{\varepsilon}\right) - \varepsilon {\mathcal T}\left(\frac{y}{\varepsilon}\right) \Big| 
-  \ln \Big| \varepsilon {\mathcal T}\left(\frac{x}{\varepsilon}\right) 
- \varepsilon {\mathcal T}\left(\frac{y}{\varepsilon}\right)^* \Big| 
+  \ln \Big| \varepsilon {\mathcal T}\left(\frac{x}{\varepsilon}\right) \Big|
\Bigl)\omega^{\varepsilon} (x) \omega^{\varepsilon} (y) \, dx \, dy.
\end{equation*}
Next we make the change of variables $X= \varepsilon {\mathcal T}(\frac{x}{\varepsilon})$ and $Y= \varepsilon {\mathcal T}(\frac{y}{\varepsilon})$ to obtain
\begin{equation} \label{R1eps}
R_1^{\varepsilon}=  \frac{1}{2\pi}\int_{B(0,\varepsilon)^c \times B(0,\varepsilon)^c } 
\Bigl( \ln | X - Y | -  \ln |X - \varepsilon^2 Y^*| +  \ln | X | \Bigl) 
f^{\varepsilon} (X) f^{\varepsilon} (Y) \, dX \, dY,
\end{equation}
where 
\begin{equation*}
f^{\varepsilon}(z) := \omega^{\varepsilon}\left(\varepsilon {\mathcal T}^{-1}\left(\frac{z}{\varepsilon}\right)\right) \, |\det (D{\mathcal T})^{-1}| \left(\frac{z}{\varepsilon}\right). 
\end{equation*}
Changing variables back, we note that
\begin{equation*}
\|f^{\varepsilon} \|_{L^1(B(0,\varepsilon)^c)}
= \|\omega^{\varepsilon} \|_{L^1(\mathcal{F}^{\varepsilon}_0 )}
\leq \| w_0 \|_{L^1(\R^{2})}.
\end{equation*}
Moreover, as $D{\mathcal T}^{-1}$ is bounded (due to \eqref{Eq:DevtT} and the regularity of ${\mathcal S}_{0}$), we have that
\begin{equation*}
\|f^{\varepsilon} \|_{L^p (B(0,\varepsilon)^c)} \leq C\|\omega^{\varepsilon} \|_{L^p (\mathcal{F}^{\varepsilon}_0 )} \leq C \| w_0 \|_{L^p(\R^{2})}.
\end{equation*} 
%
%
As $Y^*\in B(0,1/\varepsilon)$ for $Y \in B(0,\varepsilon)^{c}$, we have that $\varepsilon^2 Y^* \in B(0,\varepsilon) \subset B(0,\rho_{f^{\varepsilon}})$.
Using Lemma \ref{intln} and the fact that in \eqref{R1eps} it is enough to consider $Y \in \mbox{Supp}(f^{\varepsilon}) \subset B(0,\rho_{f^{\varepsilon}})$, we deduce that
\begin{equation*}
|R_1^{\varepsilon}| \leq 3 (C \|w_0 \|_{L^p(\R^{2})} + \ln(2\rho_{f^{\varepsilon}})
\|w_0\|_{L^1(\R^{2})}) \|w_0\|_{L^1(\R^{2})}.
\end{equation*}
Thanks to the behavior \eqref{Eq:DevtT} at infinity of ${\mathcal T}$, we know that there exists $C_0\geq \beta$ such that ${\mathcal T}(B(0,d))\subset B(0,C_0 d)$ for any $d>1$. Then,
\begin{equation*}
\mbox{Supp}(f^{\varepsilon}(t)) = \varepsilon {\mathcal T} \left(\frac{\mbox{Supp}(\omega^{\varepsilon}(t))}{\varepsilon}\right) \subset B(0,C_0 \rho^{\varepsilon}),
\end{equation*}
which involves that $\rho_{f^{\varepsilon}} \leq C_0 \rho^{\varepsilon}$, and we finally obtain
\begin{equation}\label{R1}
|R_1^{\varepsilon}(t)| \leq C_1 [1 + \ln (\rho^{\varepsilon}(t))].
\end{equation}
%
%
Using the same reasoning on $R_2^{\varepsilon}$ as for $R_1^{\varepsilon}$, we obtain
\begin{equation*}
R_2^{\varepsilon}=  \frac{1}{2\pi}\int_{\mathcal{F}^{\varepsilon}_0}   \ln \Big| \varepsilon {\mathcal T}\left(\frac{x}{\varepsilon}\right) \Big|\omega^{\varepsilon} (x) \, dx.
\end{equation*}
Hence we also deduce
\begin{equation}\label{R2}
|R_2^{\varepsilon}(t)| \leq C_2 [1 + \ln (\rho^{\varepsilon}(t))].
\end{equation}
Finally, we use that $\hat{\mathcal H}^{\varepsilon}$ is constant in time, and putting together \eqref{decompoH}, \eqref{R1} and \eqref{R2}, we get:
\begin{eqnarray*}
m |\ell^{\varepsilon} |^2(t) + \mathcal{J}_0 ({\varepsilon}r^{\varepsilon}(t))^2
&\leq& m |\ell^{\varepsilon} |^2(t) + \mathcal{J}_0 ({\varepsilon}r^{\varepsilon}(t))^2 + (X^{\varepsilon})^T \mathcal{M}_2^{\varepsilon} X^{\varepsilon}(t) \\
&\leq& m |\ell_0 |^2 + \mathcal{J}_0 ({\varepsilon}r_0)^2
+ X_0^T \mathcal{M}_2^\varepsilon X_0 -R_1^{\varepsilon}(0) 
-2\gamma R_2^{\varepsilon}(0)+R_1^{\varepsilon}(t) +2\gamma R_2^{\varepsilon}(t)\\
&\leq& C [1 + \ln (\rho^{\varepsilon}(t))].
\end{eqnarray*}
Above we used the notation $X_{0}=(\ell_{0},r_{0})$ and the boundedness of ${\mathcal M}_{2}^{\varepsilon}$ with respect to $\varepsilon$ which is a consequence of \eqref{EvoMatrice2sec}. This concludes the proof of Proposition \ref{ProRata}.
\end{proof}
%
%
%
%
%
%
\subsection{Velocity}
\label{APF}
We will use the following lemma  (see \cite[Theorem 4.1]{ift_lop_euler} and \cite[Lemma 3.5]{Cricri}).
\begin{Lemma} \label{ift_lop}
There exists a constant $C>0$ which depends only on the shape of the solid for $\varepsilon=1$ such that for any $ \omega$ smooth enough,
\begin{equation} \label{EstIFTLOP}
\|  K^{\eps}_H [\omega ]  \|_{L^{\infty} ( \mathcal{F}^{\eps}_{0}) }  \leqslant C  \|    \omega \|^{1-\frac{p'}{2}}_{ L^{1} ( \mathcal{F}^{\eps}_{0} )}   \|    \omega \|^{\frac{p'}{2}}_{ L^{p} ( \mathcal{F}^{\eps}_{0} )} .
\end{equation}
\end{Lemma}
Combining with the conservation laws \eqref{ConsOmega}, the decomposition \eqref{allo2} and the scaling laws \eqref{phi-scaling}-\eqref{phi-scaling2} we obtain that for any $t>0$,
\begin{equation}\label{infVtilde}
\| \tilde{v}^{\eps} (t,\cdot) \|_{L^{\infty}(\mathcal{F}^{\eps}_{0}) }
\leqslant C \left( | \ell^{\eps }  (t,\cdot ) | + | \varepsilon  r^{\eps}  (t,   \cdot )  |
+  \| w_{0} \|^{1-\frac{p'}{2}}_{ L^{1} ( \R^{2} )} \| w_{0} \|^{\frac{p'}{2}}_{ L^{p} ( \R^{2} )} \right).
\end{equation}
We will also use that combining
Proposition  \ref{propdefKwhole} with \eqref{ConsOmega} yields that for $p<+\infty$ (respectively for $p=+\infty$),
$\| K_{\R^{2}} [\omega^{\eps} (t, \cdot)]  \|_{C^{1-2/p} (\R^2)}$ (resp. $\| K_{\R^{2}} [\omega^{\eps} (t, \cdot)]  \|_{\mathcal{LL} (\R^2)}$) is bounded independently of $t$ and of $\eps$, where $\omega^{\eps}$ is extended by  $0$ inside ${\mathcal S}^{\varepsilon}_0$.
\subsection{Support of the vorticity}
\label{RhoBorne}
We start with the following result about $\rho^{\varepsilon}$.
\begin{Lemma}
\label{support}
For all $t \geq 0$, 
\begin{equation*}
\rho^{\varepsilon}(t) \leq \rho^{\varepsilon}(0) + \int_{0}^{t} \| v^{\varepsilon} - \ell^{\varepsilon} \|_{L^{\infty}(\R^{2} \setminus B(0,1))}.
\end{equation*}
\end{Lemma}
This lemma will be proven in Section \ref{PreuveYudoAdd}, together with the properties of solutions given by Theorem \ref{ThmYudo}. \par
\ \par
Let us now deduce an estimate on $\rho$ in terms of the time and of the initial data. We note that $\rho^{\varepsilon}(0)$ does not depend on $\varepsilon$. 
We also see that, due to \eqref{ScalingH} and \eqref{HalInfini}, $\| H^\varepsilon \|_{L^{\infty}(\R^{2} \setminus B(0,1))}$ is bounded independently of $\varepsilon$. It follows from \eqref{allo} and \eqref{infVtilde} that
\begin{equation*}
\rho^{\varepsilon}(t) \leq \rho^{\varepsilon}(0) 
+ C \int_{0}^{t} \left( |  \ell^{\eps }  (\tau, \cdot ) | + 
| \varepsilon  r^{\eps}  (\tau, \cdot ) | 
+  \| w_{0} \|^{1-\frac{p'}{2}}_{ L^{1} ( \R^{2} )} 
\| w_{0} \|^{\frac{p'}{2}}_{ L^{p} ( \R^{2})} 
+ |\gamma| \right) \, d\tau .
\end{equation*}

Using \eqref{EstNRJ}, it follows that for some constants $C_{1}$, $C_{2}>0$ depending only on $m$, ${\mathcal J}_{0}$, $\ell_{0}$, $r_{0}$, $w_{0}$, $\rho_{w_{0}}$ and the geometry for $\varepsilon=1$, we have:
\begin{equation*}
\rho^{\varepsilon}(t) \leq C_{1} + C_{2} \int_{0}^{t} (1 + \ln(\rho^{\varepsilon}(\tau)) ) \,d\tau.
\end{equation*}
Using Gronwall's lemma, we deduce that for any $T>0$, $\rho^{\varepsilon}$ is bounded on $[0,T]$ independently of $\varepsilon$.
\subsection{Main a priori bounds}
Gathering the previous estimates, we deduce the following.
\begin{Proposition} \label{Prop:APE}
For all $T>0$, $\|\rho^{\varepsilon} \|_{L^{\infty}(0,T)}$, $\| \ell^{\varepsilon} \|_{L^{\infty}(0,T)}$, $\| \varepsilon r^{\varepsilon} \|_{L^{\infty}(0,T)}$,  $\| \tilde{v}^{\eps} \|_{L^{\infty}(0,T; L^{\infty}(\mathcal{F}^{\eps}_{0})) }$ and $\| K_{\R^{2}} [\omega^{\eps}]  \|_{L^{\infty}(0,T; C^{1-2/p}  (\R^2))}$ if $p<+\infty$ (resp. $\| K_{\R^{2}} [\omega^{\eps}]  \|_{L^{\infty}(0,T;\mathcal{LL} (\R^2))}$ if $p=+\infty$) are bounded independently of $\varepsilon>0$.
\end{Proposition}
\subsection{Approximation of the velocity}
We will also use that 
\begin{equation} \label{DefCheckV}
\check{v}^{\eps}  := K_{\R^{2}}[\omega^{\eps}] + \sum_{i=1}^2 (\ell^{\eps} - K_{\R^{2}}[\omega^{\eps}]_{|x=0} )_{i} \nabla \Phi^{\eps}_{i} + r^{\eps} \nabla \Phi^{\eps}_{3},
\end{equation}
is a good approximation of  $\tilde{v}^{\eps}$ (which was introduced in \eqref{allo}).
More precisely we have the following estimate:
\begin{Proposition} \label{Pr}
As $\eps$ approaches $0^+$, we have:

\begin{equation} \label{EstCheckvTildev2}
\| \check{v}^{\eps} - \tilde{v}^{\eps} \|_{L^{\infty}(0,T;L^{2}(\partial {\mathcal S}_0^{\eps}))} +
\| \check{v}^{\eps} - \tilde{v}^{\eps} \|_{L^{\infty}(0,T ; H^{1/2}_{\varepsilon}({\mathcal F}_0^{\eps}))}=  o(\eps^{1/2}),
\end{equation}
where
\begin{equation} \nonumber
\| \cdot \|_{H^{1/2}_{\varepsilon}({\mathcal F}_0^{\eps})} := \| \cdot \|_{\dot{H}^{1/2}({\mathcal F}_0^{\eps})} + \varepsilon^{-1/2} \| \cdot \|_{L^{2}({\mathcal F}_0^{\eps})}.
\end{equation}
\end{Proposition}
\begin{proof}
One checks that
\begin{equation} \label{CheckvTildev}
\left\{ \begin{array}{l} 
\curl( \check{v}^{\eps} - \tilde{v}^{\eps}) =0,  \quad   \text{for}  \ x\in  \mathcal{F}^{\eps}_{0} , \\ 
\div( \check{v}^{\eps} - \tilde{v}^{\eps}) =0,  \quad  \text{for}  \ x\in  \mathcal{F}^{\eps}_{0}  , \\
\int_{\partial {\mathcal S}_0^{\eps}} ( \check{v}^{\eps} - \tilde{v}^{\eps})\cdot \tau \, ds =0, \\
( \check{v}^{\eps} - \tilde{v}^{\eps}) \cdot n =  g^{\eps} , \quad   \text{for}  \ x\in \partial \mathcal{S}^{\eps}_0  , \\
\check{v}^{\eps} - \tilde{v}^{\eps} \rightarrow 0 \quad  \text{as}  \ x \rightarrow  \infty ,
\end{array} \right.
\end{equation}
with
\begin{equation} \label{DefGeps}
g^{\eps} := (K_{\R^{2}}[\omega^{\eps}] -K_{\R^{2}}[\omega^{\eps}]_{|x=0})\cdot n .
\end{equation} 
As a consequence there exists $\Psi^{\eps} $ such that 
\begin{equation} \label{TropConnecte}
\check{v}^{\eps} - \tilde{v}^{\eps} = \nabla \Psi^{\eps},
\end{equation}
and 
\begin{equation}
 \label{neuneu}
\left\{ \begin{array}{l} 
\Delta \Psi^{\eps} = 0 ,  \quad   \text{for}  \ x\in  \mathcal{F}^{\eps}_{0} , \\ 
\partial_n  \Psi^{\eps} = g^{\eps} ,  \quad  \text{for}  \ x\in \partial \mathcal{S}^{\eps}_0   , \\
 \Psi^{\eps} \rightarrow 0 \quad  \text{as}  \ x \rightarrow  \infty .
 \end{array} \right.
\end{equation}
We now use a dilatation argument and the following classical result (see for instance \cite{Kikuchi83}).
\begin{Lemma} \label{LemmeElliptiqueDeBase}
There exists $C>0$ such that for any $g$ in $L^2 ( \partial \mathcal{S}_0 )$ satisfying 
\begin{equation} \label{intgzero}
\int_{\partial {\mathcal S}_0} g(s) ds = 0,
\end{equation}
there is only one  solution $\Psi$ in $H^\frac{3}{2} ( \mathcal{F}_{0} )$ solution of 
\begin{equation}
 \label{neuneu0}
\left\{ \begin{array}{l} 
\Delta \Psi = 0 ,  \quad   \text{for}  \ x\in  \mathcal{F}_{0} , \\ 
\partial_n  \Psi = g ,  \quad  \text{for}  \ x\in \partial \mathcal{S}_0   , \\
 \Psi \rightarrow 0 \quad  \text{as}  \ x \rightarrow  \infty .
 \end{array} \right.
\end{equation}
given as the potential layer 
\begin{eqnarray*}
\Psi  (x) = - \int_{\partial {\mathcal S}_0} G_{ \mathcal{F}_{0}} (x,y) g(y) ds(y) , 
\end{eqnarray*}
where $G_{ \mathcal{F}_{0}}$ stands for the Green's function associated to the exterior domain $\mathcal{F}_{0}$ with Dirichlet boundary condition, and 
\begin{equation}
 \label{neuneu0esti}
 \|  \Psi   \|_{H^{3/2}({\mathcal F}_0)} \leq C \|  g  \|_{L^{2}( \partial  {\mathcal S}_0)} .
 \end{equation}
\end{Lemma}
Note that by a classical trace lemma, one has for some constant $C>0$:
\begin{equation*}
\|  \nabla \Psi \cdot \tau  \|_{L^{2}(\partial {\mathcal S}_0)} \leq C \|  \Psi   \|_{H^{3/2}({\mathcal F}_0)}.
\end{equation*}
Now we use the change of variables:
\begin{equation*}
\Psi^{\eps}  (x) = \eps \Psi (x/\eps) , \quad g^{\eps}  (x) = g (x/\eps) ,
\end{equation*}
with
\begin{equation*}
\| \nabla \Psi^{\varepsilon} \cdot \tau \|_{L^{2}(\partial {\mathcal S}^{\varepsilon}_{0})} = \sqrt{\varepsilon} \| \nabla \Psi \cdot \tau \|_{L^{2}(\partial {\mathcal S}_{0})}, \ \ 
\| \nabla \Psi^{\varepsilon} \|_{H_{\varepsilon}^{1/2}({\mathcal F}^{\varepsilon}_{0})} = \sqrt{\varepsilon} \| \nabla \Psi \|_{H_{1}^{1/2}({\mathcal F}_{0})} 
\ \text{ and } \ 
\| g^{\varepsilon} \|_{L^{2}(\partial{\mathcal S}^{\varepsilon}_{0})} = \sqrt{\varepsilon} \| g \|_{L^{2}(\partial {\mathcal S}_{0})}.
\end{equation*}
We apply Lemma \ref{LemmeElliptiqueDeBase} on $(\Psi,g)$ (we note that $g$ satisfies \eqref{intgzero} because $\div(K_{\R^{2}}[\omega^{\eps}] -K_{\R^{2}}[\omega^{\eps}]_{|x=0}) =0$), and infer that for $0< \eps \leq 1$, 
\begin{equation*}
\| (\check{v}^{\eps} - \tilde{v}^{\eps})\cdot \tau \|_{L^{\infty}(0,T;L^{2}(\partial {\mathcal S}_0^{\eps}))} +
\| \check{v}^{\eps} - \tilde{v}^{\eps} \|_{L^{\infty}(0,T;H_{\varepsilon}^{1/2}({\mathcal F}_0^{\eps}))} 
 \leq C \| g^{\varepsilon} \|_{L^{\infty}(0,T;L^{2}(\partial {\mathcal S}_0^{\eps}))}.
\end{equation*}
Finally we use the uniform H\"older estimate on $K_{\R^{2}}[\omega^{\eps}]$ given by Proposition \ref{Prop:APE} to deduce
\begin{eqnarray*}
\| (\check{v}^{\eps} - \tilde{v}^{\eps})\cdot n \|_{L^{\infty}(0,T;L^{2}(\partial {\mathcal S}_0^{\eps}))}
= \| g^{\varepsilon} \|_{L^{\infty}(0,T;L^{2}(\partial {\mathcal S}_0^{\eps}))} 
&\leq & C \| K_{\R^{2}}[\omega^{\eps}] -K_{\R^{2}}[\omega^{\eps}]_{|x=0} \|_{L^{\infty}(0,T;L^{2}(\partial {\mathcal S}_0^{\eps}))} \\
&=& o(\varepsilon^{1/2}) .
\end{eqnarray*}
This gives the desired conclusion and ends the proof of Proposition \ref{Pr}.
\end{proof} 
%
%
%
%
%
%
%
%
%
%
%
\section{Pressure force}
\label{PF}

The aim of this section is to study the pressure force/torque acting on the body:
\begin{eqnarray*}
F^\eps  (t) :=  \left( \int_{ \partial \mathcal{S}^\eps_0} q^\eps n \, ds, \
 \int_{ \partial \mathcal{S}^\eps_0} q^\eps x^{\perp} \cdot n \,ds \right) .
\end{eqnarray*}
To convert the previous boundary integrals into distributed integrals, we first use Green's formula and the functions $\Phi^\eps_{i}$ defined in \eqref{t1.3sec}-\eqref{t1.5sec} to write
\begin{eqnarray*}
F^\eps  (t) &=&  \left( \int_{\mathcal{F}^{\eps}_{0}}  \nabla q^\eps (x) \cdot  \nabla   \Phi^\eps_i dx\right)_{i=1,2,3} .
\end{eqnarray*}
That there is no contribution coming from infinity is justified by \eqref{Decroiq}. \par
Using  the following equality for two vector fields $a$ and $b$ in $\R^2$:
\begin{equation}
\label{vect}
\nabla(a\cdot b)=a\cdot \nabla b + b \cdot \nabla a - (a^\perp \curl b + b^\perp \curl a), 
\end{equation}
the equation (\ref{Euler11}) reads as follows
\begin{equation} \label{EqEulTrans}
\frac{\partial v^{\eps}}{\partial t}+
[ v^{\eps}-\ell^{\eps}-r^{\eps} x^\perp ]^\perp \omega^{\eps}
+ \nabla \frac{1}{2} (v^{\eps})^2 
- \nabla ( (\ell^{\eps} + r^{\eps} x^\perp)\cdot v^{\eps} )
+ \nabla q^{\eps} =0 .
\end{equation}
Plugging the decomposition \eqref{allo} into the previous equation, we find
\begin{equation} \label{voila}
\frac{\partial v^{\eps}}{\partial t}+ [ v^{\eps}-\ell^{\eps}-r^{\eps} x^\perp ]^\perp \omega^{\eps}
+ \nabla ({\mathcal Q}^{\eps}  + q^{\eps} )= 0 ,
\end{equation}
\begin{equation}
\label{voila2}
{\mathcal Q}^{\eps} :=  \frac{1}{2} |\tilde{v}^{\eps}|^2  + \gamma  (\tilde{v}^{\eps} 
- (\ell^{\eps} + r^{\eps} x^\perp))\cdot  H^{\eps} + \frac{1}{2} \gamma^2 |H^{\eps} |^2 
- (\ell^{\eps} + r^{\eps} x^\perp)\cdot \tilde{v}^{\eps}  .
\end{equation}
Using \eqref{voila}-\eqref{voila2} we get the following decomposition of $F^\eps$, for $i=1,2,3$:
\begin{eqnarray*}
- F^\eps_i  (t) &=& A_i^\eps +  B_i^\eps + C_i^\eps,
\end{eqnarray*}
where 
\begin{eqnarray*}
A_i^\eps &:=&  \int_{  \mathcal{F}^\eps_{0}  }  \partial_{t}  v^\eps  \cdot  \nabla   \Phi^\eps_i (x) \, dx , \\
B_i^\eps &:=& \int_{  \mathcal{F}^\eps_{0}  } \omega^{\eps}[ v^{\eps}-\ell^{\eps}-r^{\eps} x^\perp ]^\perp \cdot  \nabla   \Phi^\eps_i (x) \, dx , \\ 
C_i^\eps &:=& \int_{\partial  \mathcal{S}^\eps_{0}  } {\mathcal Q}^{\eps} K_i \, ds.
\end{eqnarray*}
For the last term, we used again Green's formula. We underline that there is no contribution from the infinity since each term in ${\mathcal Q}^{\eps}$ is (at least) bounded as $|x| \rightarrow +\infty$, while the normal derivative of $\Phi^{\varepsilon}_{i}$ over large circles satisfies $\partial_{n} \Phi^{\varepsilon}_{i} = {\mathcal O}(1/|x|^{2})$. \par
\ \par
In the rest of this section, we study the limit as $\eps$ goes to zero of all these terms.
\subsection{Treatment of the first term}
Let us examine $A_i^\eps$ for $i=1,2,3$. As $\div v^\eps =0$ for all time, using \eqref{t1.3sec}-\eqref{t1.4sec} we deduce that
\begin{eqnarray*}
A_i^\eps =  \int_{  \partial \mathcal{S}^\eps_{0}  } \partial_{t}  v^\eps  \cdot  n \, \Phi^\eps_i (x) \, ds .
\end{eqnarray*}
Now using the boundary condition \eqref{Euler13}, \eqref{EvoMatrice2sec} and Green's formula we obtain
\begin{eqnarray*}
(A_i^\eps )_{i=1,2,3} = \mathcal{M}^{\eps}_2 
\begin{pmatrix} \ell_{1}^{\eps} \\ \ell_{2}^{\eps} \\r^{\eps} \end{pmatrix} '(t)  .
\end{eqnarray*}
These terms will be put on the left hand side of the solid equations (see also \eqref{CalcPression2} in the Section concerning the Cauchy problem).
\subsection{Limit for the second term}
The second term will have no contribution in the limit, as the following proposition shows.
\begin{Proposition} \label{LimBi}
As $\varepsilon \rightarrow 0^{+}$, one has:
\begin{equation*}
B_{i}^{\varepsilon} \longrightarrow 0 \  \text{ for } i=1,2 \text{ and }
\frac{B_{3}^{\varepsilon}}{\varepsilon} \longrightarrow 0 \text{ for } i=3.
\end{equation*}
\end{Proposition}
\begin{proof}[Proof of Proposition \ref{LimBi}]
According to \eqref{allo}, we cut $B^{\varepsilon}_{i}$ in two parts: $B_{i}^{\varepsilon} = \hat{B}_{i}^{\varepsilon} + \check{B}_{i}^{\varepsilon}$ with
\begin{gather}
\label{CVB1}
\hat{B}_{i}^{\varepsilon} := \int_{  \mathcal{F}^\eps_{0}  } \omega^{\eps}[ \tilde{v}^{\eps}-\ell^{\eps}-r^{\eps} x^\perp ]^\perp \cdot  \nabla   \Phi^\eps_i (x) \, dx , \\
\label{CVB2}
\check{B}_{i}^{\varepsilon} := \int_{  \mathcal{F}^\eps_{0}  } \gamma \omega^{\eps}[H^{\varepsilon}]^\perp \cdot  \nabla   \Phi^\eps_i (x) \, dx .
\end{gather}
\noindent
The following estimates are uniform with respect to $t \in [0,T]$. \par
\ \par
\noindent
{\bf 1.} Let us begin with $i=1,2$. In this case we have
\begin{align*}
|\hat{B}_i^\eps| &= \left| \int_{ \mathcal{F}^\eps_{0} \cap B(0,\rho^{\varepsilon}) } [ \tilde{v}^{\eps}-\ell^{\eps}-r^{\eps} x^\perp ]^\perp \omega^{\eps}\cdot  \nabla   \Phi^\eps_i (x) \, dx \right| \\ 
&\leq \| \tilde{v}^{\eps}-\ell^{\eps}-r^{\eps} x^\perp \|_{L^{\infty}(\mathcal{F}^\eps_{0} \cap B(0,\rho^{\varepsilon}))} \| \omega^{\eps}\|_{L^{p}}  \| \nabla   \Phi^\eps_i \|_{L^{p'}(\mathcal{F}^\eps_{0} \cap B(0,\rho^{\varepsilon}))},
\end{align*}
when $p<+\infty$. In the case $p=+\infty$, we have in particular $w_{0} \in L^{3}_{c}({\mathcal F}_{0})$ and use the inequalities written here for $p=3$. 
Using \eqref{phi-scaling} and recalling \eqref{ComportementPhii}, we see that for $r>1$
\begin{equation}
\label{thenry}
\| \nabla   \Phi^\eps_i \|_{L^{r}(\mathcal{F}^\eps_{0} \cap B(0,\rho^{\varepsilon}))} 
= \varepsilon^{2/r} \| \nabla   \Phi^{1}_i \|_{L^{r}(\mathcal{F}_{0} \cap B(0,\rho^{\varepsilon}/\varepsilon))} 
\leq \varepsilon^{2/r} \| \nabla \Phi^{1}_i \|_{L^{r}(\mathcal{F}_{0})}.
\end{equation}
On the other side, due to \eqref{infVtilde} and to Proposition \ref{Prop:APE}, we have
\begin{equation*}
|\rho^{\varepsilon}| +  \| \tilde{v}^{\varepsilon} \|_{L^{\infty}(\mathcal{F}^\eps_{0} \cap B(0,\rho^{\varepsilon}))}  + | \ell^{\varepsilon} | + \varepsilon |r^{\varepsilon}| \leq C.
\end{equation*}
It follows that
\begin{equation*}
\hat{B}_i^\eps = {\mathcal O} \left(\varepsilon^{\frac{2}{p'} -1}\right) \rightarrow  0 \text{ as } \varepsilon \rightarrow 0^{+}.
\end{equation*}
Concerning $\check{B}_i^\eps$, we write
\begin{align*}
|\check{B}_i^\eps| &\leq |\gamma| \| H^{\eps}\|_{L^{q}(\mathcal{F}^\eps_{0} \cap B(0,\rho^{\varepsilon}))} \| \omega^{\eps}\|_{L^{p}}  \| \nabla   \Phi^\eps_i \|_{L^{r}(\mathcal{F}^\eps_{0} \cap B(0,\rho^{\varepsilon}))} ,
\end{align*}
with $ \frac{1}{q}=   \frac{1}{2} ( \frac{1}{p'} + \frac{1}{2} )$ and $ \frac{1}{r}=   \frac{1}{2} ( \frac{1}{p'} - \frac{1}{2} )$.
We use the fact that $\| H^{\eps}\|_{L^{q}(\mathcal{F}^\eps_{0} \cap B(0,\rho^{\varepsilon}))}$ is bounded independently of $\varepsilon$ (see \eqref{ScalingH} and \eqref{HalInfini} and observe that $q<2$) and once again \eqref{thenry}. Hence we also have
\begin{equation*}
\check{B}_i^\eps \longrightarrow  0 \text{ as } \varepsilon \rightarrow 0^{+}.
\end{equation*}
{\bf 2.} Let us now turn to the case $i=3$. In that case, the scaling of $\nabla \Phi_{3}^{\varepsilon}$ is not the same. But using \eqref{phi-scaling2}, we see that the situation is actually better, in the sense that using the same estimates as before, we get an additional power of $\varepsilon$. Then \eqref{CVB2} follows.
\end{proof}
\subsection{Limit for the third term}
We decompose $C_{i}^\eps $ into 
\begin{eqnarray}
\label{Cia}
C_{i,a}^\eps &=& \frac{1}{2} \int_{\partial  \mathcal{S}^\eps_{0}  } |\tilde{v}^{\eps}|^2  K_i \, ds, \\ 
\label{Cib}
C_{i,b}^\eps &=& \gamma  \int_{\partial  \mathcal{S}^\eps_{0}  } (\tilde{v}^{\eps} - (\ell^{\eps} + r^{\eps} x^\perp))\cdot  H^{\eps}  K_i \, ds, \\
\label{Cic}
C_{i,c}^\eps &=& \frac{\gamma^2}{2}   \int_{\partial  \mathcal{S}^\eps_{0} } |H^{\eps}|^2  K_i \, ds, \\
\label{Cid}
C_{i,d}^\eps &=&  - \int_{\partial  \mathcal{S}^\eps_{0}  } (\ell^{\eps} + r^{\eps} x^\perp)\cdot \tilde{v}^{\eps}  K_i \, ds.
\end{eqnarray}
{\bf 1.} We first tackle the terms $C_{i,a}^\eps$ and $C_{i,d}^\eps$ which are the easiest ones. One easily sees that for $i=1,2$:
\begin{eqnarray*}
| C_{i,a}^\eps | +|  C_{i,d}^\eps| \leqslant C \varepsilon (\| \tilde{v}^{\varepsilon} \|_{L^{\infty}}^{2} + |\ell^{\varepsilon}|^{2} + |\varepsilon r^{\varepsilon}|^{2} ),
\end{eqnarray*}
and that for $i=3$:
\begin{eqnarray*}
| C_{3,a}^\eps | +|  C_{3,d}^\eps| \leqslant C \varepsilon^{2} (\| \tilde{v}^{\varepsilon} \|_{L^{\infty}}^{2} + |\ell^{\varepsilon}|^{2} + |\varepsilon r^{\varepsilon}|^{2} ).
\end{eqnarray*}
We conclude with Proposition \ref{Prop:APE} that these terms tend to zero as $\varepsilon \rightarrow 0^{+}$, as ${\mathcal O}(\varepsilon^{2})$ when $i=3$. \par
\ \par
\noindent
{\bf 2.} We turn to $C_{i,c}^{\varepsilon}$. We will make use of the following classical Blasius' lemma (see for instance \cite{MP} and \cite[Problem 4.3]{Childress}), which we prove in the appendix for the sake of self-containedness. 
\begin{Lemma}
\label{blasius}
Let  $\mathcal{C}$ be a smooth Jordan curve, $f:=(f_1 , f_2)$ and $g:=(g_1 ,g_2 )$ two smooth tangent vector fields on $\mathcal{C}$. 
Then 
\begin{eqnarray} 
\label{bla1}
\int_{ \mathcal{C}} (f  \cdot g) n \, ds =  i \left( \int_{ \mathcal{C}} (f_1 - if_2) (g_1 - i g_2) \, dz \right)^* , \\
\label{bla2}
\int_{ \mathcal{C}} (f  \cdot g) (x^{\perp} \cdot n)  \,ds =  \Re \left( \int_{ \mathcal{C}} z (f_1 - if_2) (g_1 - i g_2) \, dz \right).
\end{eqnarray}
where $(\cdot )^*$ denotes the complex conjugation.
\end{Lemma}
\noindent
We now apply Lemma \ref{blasius} and use \eqref{HSeriesLaurent} and Cauchy's Residue Theorem. We deduce directly that $C^{\varepsilon}_{1,c}=C^{\varepsilon}_{2,c}=C^{\varepsilon}_{3,c}=0$ . \par
\ \par
\noindent
{\bf 3.} Let us finally turn to the main term, that is $C^{\varepsilon}_{i,b}$. Let us prove the following.
\begin{Proposition} \label{PropCib}
One has for $i=1,2$
\begin{equation} \label{Cib12}
C_{i,b}^{\varepsilon} = \gamma ( K_{\R^2}[\omega^\eps] (t,0) -  \ell^{\eps} )^\perp +\varepsilon r^{\varepsilon} \gamma \xi + o(1),
\end{equation}
and 
\begin{equation} \label{Cib3}
C_{3,b}^{\varepsilon} = \gamma \varepsilon \, \zeta \cdot ( K_{\R^2}[\omega^\eps] (t,0) -  \ell^{\eps} )  + o(\varepsilon),
\end{equation}
where $\xi$ and $\zeta$ are defined in $\R^{2}\simeq \C$ by
\begin{gather*}
\xi: = \left(\int_{\partial  \mathcal{S}_{0}  } \overline{z} (H^1_{1} -i H^1_{2}) \, dz \right)^{*}, \\
\zeta: = \int_{ \partial  \mathcal{S}_{0}  } (H^1_{1} -i H^1_{2})  z \, dz.
\end{gather*}
\end{Proposition}
\begin{proof}[Proof of Proposition \ref{PropCib}]
We introduce 
\begin{eqnarray*}
\overline v^\eps (t,x) := K_{\R^2}[\omega^\eps] (t,0) 
+ \sum_{i=1}^2 (\ell^\eps_i (t)  - K_{\R^2}[\omega^\eps]_i (t,0)) \nabla \Phi_i^\eps (x)  + r^{\varepsilon} (t) \nabla \Phi_3^\eps  (x) \ \text{ in }\  \R^{+} \times{\mathcal F}_{0}^{\varepsilon},
\end{eqnarray*}
which will give a good approximation of $\check{v}^{\eps}$  on $\partial \mathcal{S}^\eps_{0}$ (compare to \eqref{DefCheckV}). 
Here $\omega^{\eps}$ is again extended by  $0$ inside ${\mathcal S}^{\varepsilon}_0$. 
Note in particular that one has
\begin{equation} \label{OvVauBord}
\overline v^\eps \cdot n = \tilde v^\eps \cdot n = (\ell^{\eps} + r^{\eps} x^\perp) \cdot n \ \text{ on } \partial \mathcal{S}^\eps_{0}. 
\end{equation}

The two steps in estimating $C^{\varepsilon}_{i,b}$ consists in computing the integral $C^{\varepsilon}_{i,b}$ in \eqref{Cib} when $\tilde{v}^{\eps}$ is replaced with $\overline v^\eps$, and then to show that the error of this replacement is small as $\varepsilon \rightarrow 0^{+}$. \par
\ \par
\noindent
{\bf a.} Denote
\begin{equation*}
\hat{C}_{i,b}^\eps = \gamma  \int_{\partial  \mathcal{S}^\eps_{0}  } (\overline{v}^{\eps} - (\ell^{\eps} + r^{\eps} x^\perp))\cdot  H^{\eps}  K_i \, ds,
\end{equation*}
and
\begin{equation} \label{Defuv}
\underline v^\eps(t,x) := \overline v^\eps(t,\varepsilon x) - (\ell^{\eps} + \varepsilon r^{\eps} x^\perp)  \ \text{ in }\  \R^{+} \times {\mathcal F}_{0}.
\end{equation}
By a direct scaling argument (see \eqref{ScalingH}), we deduce that
\begin{equation*}
\hat{C}_{i,b}^\eps = \hat{C}_{i,b} \ \text{ with }\ 
\hat{C}_{i,b} := \gamma  \int_{\partial  \mathcal{S}_{0}  } \underline{v}^{\varepsilon} \cdot  H^1 n_i \, ds \ \text{ for } i=1,2,
\end{equation*}
and
\begin{equation*}
\hat{C}_{3,b}^\eps = \varepsilon \hat{C}_{3,b} \ \text{ with }\ 
\hat{C}_{3,b} := \gamma  \int_{\partial  \mathcal{S}_{0}  } \underline{v}^{\varepsilon} \cdot  H^1 (x^{\perp}\cdot n) \, ds.
\end{equation*}
We remark that $\overline{v}^{\eps}-(\ell^{\eps} + r^{\eps} x^\perp)$ is a smooth vector field, tangent to $\partial {\mathcal S}^{\varepsilon}_{0}$, so that $\underline{v}^{\varepsilon}$ is tangent to $\partial {\mathcal S}_{0}$ . Hence, we get by Lemma \ref{blasius} that
\begin{eqnarray}
\label{blasius2}
(\hat{C}_{1,b}, \, \hat{C}_{2,b}) 
=  \gamma \int_{\partial  \mathcal{S}_{0}  } (\underline{v}^\eps \cdot  H^1)  n \, ds 
= i \gamma \left( \int_{ \partial  \mathcal{S}_{0}  }( \underline v^\eps_1  -i \underline v^\eps_2 )  (H^1_{1} -i H^1_{2}) \, dz \right)^*, \\
\label{blasius3}
\hat{C}_{3,b} 
= \gamma \int_{\partial  \mathcal{S}_{0}  } \underline v^\eps \cdot  H^1  (x^{\perp}\cdot n)  \, ds
=  \gamma \, \Re \left( \int_{ \partial  \mathcal{S}_{0}  }( \underline v^\eps_1  -i \underline v^\eps_2 )  (H^1_{1} -i H^1_{2})  z \, dz \right).
\end{eqnarray}
Let us denote
\begin{equation*}
\underline v^\eps_\infty :=  K_{\R^2}[\omega^\eps] (t,0) -  \ell^{\eps}.
\end{equation*}
For what concerns $ (H^{\eps}_{1} -i H^{\eps}_{2}) $ we have \eqref{HSeriesLaurent}. Concerning $\underline{v}^{\varepsilon}$, due to \eqref{ComportementPhii} and \eqref{Defuv}, and using
\begin{equation*}
(x^{\perp})_{1}-i (x^{\perp})_{2} = -i {(x_{1}+ix_{2})}^{*},
\end{equation*}
we have that
\begin{eqnarray*}
\underline v_{1}^\eps - i\underline v_{2}^\eps =  i \varepsilon r^{\eps} \overline z +\underline v^\eps_{\infty ,1} - i \underline v^\eps_{\infty ,2} + {\mathcal O}(1/ | z |^{2}).
\end{eqnarray*}
\ \par
\noindent
$\bullet$ We first study $(\hat{C}_{1,b}, \, \hat{C}_{2,b})$. 
Using Cauchy's residue theorem, we deduce that for $i=1,2$:
\begin{eqnarray}
\label{blasius4}
i\left(\int_{\partial  \mathcal{S}_{0}  } (\underline v_{1}^\eps - i\underline v_{2}^\eps -i\varepsilon r^{\eps} \overline z)  (H^1_{1} -i H^1_{2}) \, dz \right)^{*}
&=& i   ( \underline v^\eps_{\infty ,1} - i \underline v^\eps_{\infty ,2})^* \\
&=& {\underline{v}^\eps_\infty}^\perp .
\end{eqnarray}
With the definition of $\xi$ we deduce
\begin{equation} \label{Lrot}
(\hat{C}_{1,b}, \, \hat{C}_{2,b}) 
= \gamma {\underline{v}^\eps_\infty}^\perp +\varepsilon r^{\varepsilon} \gamma \xi.
\end{equation}
\noindent
$\bullet$ We now consider $\hat{C}_{3,b}$. We will use the following lemma, proved in the appendix.
\begin{Lemma} \label{Blabla}
\begin{equation*}
\Im \left( \int_{ \partial  \mathcal{S}_{0}  }\overline{z} (H^1_{1} -i H^1_{2})  z \, dz \right)=0.
\end{equation*}
\end{Lemma}
\noindent
It follows from this lemma that the term $r^{\eps} x^\perp$ does not intervene in  $\hat{C}_{3,b}$.
By Cauchy's residue theorem and using the definition of $\zeta$ we deduce that
\begin{equation*}
\Re \left(\int_{\partial  \mathcal{S}_{0}  } (\underline v_{1}^\eps - i\underline v_{2}^\eps +\varepsilon r^{\eps} \overline z)  (H^1_{1} -i H^1_{2}) z \, dz \right) = \zeta \cdot \underline v^\eps_\infty.
\end{equation*}

%
%
%
\ \par
\noindent
{\bf b.} Let us now establish that 
\begin{equation} \label{Pitipiti}
\int_{\partial  \mathcal{S}^\eps_{0}  } (\overline v^\eps - \tilde{v}^{\eps} ) \cdot H^{\eps} n_{i} \, ds = o( 1),
\end{equation}
and that
\begin{equation} \label{Pitipiti2}
\int_{\partial  \mathcal{S}^\eps_{0}  } (\overline v^\eps - \tilde{v}^{\eps} ) \cdot H^{\eps} (x^{\perp}\cdot n) \, ds = o(\eps).
\end{equation}
\noindent
On one side, it is straightforward using the  H\"older  estimate on $K_{\R^{2}}[\omega^{\eps}]$  given by Proposition \ref{Prop:APE}  to infer that
\begin{equation} \nonumber
| \check{v}^{\eps} - \overline{v}^{\eps}| = o( 1) \  \text{ uniformly on } (0,T) \times \partial {\mathcal S}_0^{\eps}.
\end{equation}
One the other side from Proposition \ref{Pr} we have that
$\| \check{v}^{\eps} - \tilde{v}^{\eps} \|_{L^{\infty}(0,T;L^{2}(\partial {\mathcal S}_0^{\eps}))}  = o(\eps^{1/2})$. 
Estimates \eqref{Pitipiti} and \eqref{Pitipiti2} follow by the Cauchy-Schwarz inequality.
\end{proof} 
\subsection{Conclusion}
Putting together all the results established in this section, we can state the following proposition.
\begin{Proposition} \label{PropositionPression}
The pressure force/torque can be written:
\begin{equation*}
\begin{pmatrix} F_{1}^\eps \\ F_{2}^\eps \\ F_{3}^\eps \end{pmatrix} 
= - {\mathcal M}^{\varepsilon}_{2} \begin{pmatrix} \ell^\eps \\ r^{\varepsilon} \end{pmatrix} ' 
+ \gamma \begin{pmatrix} (\ell^\eps - K_{\R^2}[\omega^\eps](t,0))^\perp -\varepsilon r^{\varepsilon} \xi \\
\varepsilon  \, \zeta \cdot (\ell^\eps - K_{\R^2}[\omega^\eps](t,0)) \end{pmatrix}
+ \begin{pmatrix} R_{1}^\eps \\ R_{2}^\eps \\ \varepsilon R_{3}^\eps \end{pmatrix} ,
\end{equation*}
with
\begin{equation*}
R_{i}^\eps \longrightarrow 0 \ \text{ in } L^{\infty}(0,T) \ \text{ as } \ \varepsilon \rightarrow 0^{+}.
\end{equation*}
%
%
\end{Proposition}
\section{Passage to the limit}
\label{Passage}
\subsection{Compactness}
{\bf 1.} {\it Compactness for the solid velocity.} We begin by obtaining compactness on the solid linear and angular velocities in the original frame. Using Proposition \ref{PropositionPression} and \eqref{Solide11}-\eqref{Solide12}, we obtain
\begin{equation*}
\mathcal{M}^{\eps} \begin{pmatrix} \ell^{\varepsilon} \\ r^{\varepsilon} \end{pmatrix}'
=  \gamma \begin{pmatrix} (\ell^\eps - K_{\R^2}[\omega^\eps](t,0))^\perp - \varepsilon r^{\varepsilon} \xi \\
\varepsilon  \, \zeta \cdot (\ell^\eps - K_{\R^2}[\omega^\eps](t,0)) \end{pmatrix}
+\begin{pmatrix} -mr^{\varepsilon} (\ell^{\varepsilon})^{\perp} \\
0 \end{pmatrix}
+ \begin{pmatrix} R_{1}^\eps \\ R_{2}^\eps \\ \varepsilon R_{3}^\eps \end{pmatrix} .
\end{equation*}
We multiply by ${\mathcal M}^{\varepsilon}_{1}(\mathcal{M}^{\eps})^{-1}$; using \eqref{EvoMatrice2sec}, Proposition \ref{Prop:APE} and \eqref{ConsOmega}, and simplifying by $\varepsilon$ the second equation we deduce that
\begin{gather}
\label{lprime}
m (\ell^{\varepsilon})' = \gamma (\ell^{\varepsilon} - K_{\R^2}[\omega^\eps](t,0))^{\perp} - (\varepsilon r^{\varepsilon}) \gamma \xi - m r^{\varepsilon} (\ell^{\varepsilon})^{\perp} + \tilde{R}_{1}^{\varepsilon} \\
\label{rprime}
{\mathcal J}_{0} (\varepsilon r^{\varepsilon})' = \gamma \zeta\cdot(\ell^{\varepsilon} - K_{\R^2}[\omega^\eps](t,0)) + \tilde{R}_{2}^{\varepsilon},
\end{gather}
with
\begin{equation} \label{CVRTilde}
\tilde{R}_{1}^{\varepsilon}, \tilde{R}_{2}^{\varepsilon} \longrightarrow 0 \ \text{ in } L^{\infty}(0,T) \ \text{ as } \ \varepsilon \rightarrow 0^{+}.
\end{equation}
Going back to the original velocity by using \eqref{chgtvar} and \eqref{chgtvar2} we deduce:
\begin{gather}\label{lprime2}
m (h^{\varepsilon})'' = \gamma ((h^{\varepsilon})' - K_{\R^{2}}[w^{\varepsilon}](t,h^{\varepsilon}))^{\perp} - \gamma (\varepsilon r^{\varepsilon}) Q^{\varepsilon}(t) \xi + Q^{\varepsilon}(t) \tilde{R}_{1}^{\varepsilon}, \\
\label{rprime2}
{\mathcal J}_{0} (\varepsilon r^{\varepsilon})' = \gamma \zeta\cdot Q^{\varepsilon}(t)^T ((h^{\varepsilon})' - K_{\R^2}[w^\eps](t,h^{\varepsilon})) + \tilde{R}_{2}^{\varepsilon},
\end{gather}
We used the fact that the Biot-Savart law in the plane \eqref{BSR2} commutes with translations and rotations. \par
Thanks to Proposition \ref{Prop:APE} we have that 
$K_{\R^2}[\omega^\eps](t,0)$ is bounded in $L^{\infty}(0,T)$ as $\varepsilon \rightarrow 0^{+}$.
Now since the right hand sides of \eqref{rprime} and \eqref{lprime2} are bounded in $L^{\infty}(0,T)$ (due to Proposition \ref{Prop:APE}), we infer that for some subsequence $(\varepsilon_{n})$, $\varepsilon_{n} \rightarrow 0^{+}$ of the parameter $\varepsilon$, we have
\begin{gather} \label{CVh}
h^{\varepsilon_{n}} \cvwstar h \ \text{ in } W^{2,\infty}(0,T), \\
\label{CVr}
\varepsilon_{n} r^{\varepsilon_{n}}\cvwstar R \ \text{ in } W^{1,\infty}(0,T).
\end{gather}
\ \\
{\bf 2.} {\it Compactness for the fluid velocity.} Let us now obtain some compactness for the fluid vorticity in the original frame, and for the velocity it generates via the Biot-Savart law. We obtain a convergence along a subsequence of $(\varepsilon_{n})$; to simplify the notations we will still call it $(\varepsilon_{n})$. \par
We extend ${w}^{\varepsilon} (t,\cdot)$ by $0$ inside ${\mathcal S}^{\varepsilon} (t)$.
Using the a priori estimate \eqref{ConsOmega}, we deduce that, up to a subsequence of $(\varepsilon_{n})$, one has, for some ${w} \in  L^{\infty}(0,T;  L^{p}(\R^2) )$:
\begin{equation} \label{CVTildeW}
{w}^{\varepsilon_{n}} \cvwstar {w} \ \text{weakly} \ \text{ in } L^{\infty}(0,T;  L^{p}(\R^2) ) \ \text{ as } n \rightarrow +\infty.
\end{equation}
Also, using \eqref{ConsOmega} and Proposition \ref{propdefKwhole}, we deduce that $K_{\R^{2}}[ w^{\eps}]$ is bounded in $L^{\infty}(0,T;W^{1,p}(\R^{2}))$ as $\varepsilon \rightarrow 0^{+}$ for $p<+\infty$ (resp. in $L^{\infty}(0,T;{\mathcal{LL}}(\R^{2}))$ if $p=+\infty$). We extend $\omega^{\varepsilon}$ by $0$ inside ${\mathcal S}^{\varepsilon} (t)$.
Then it is not difficult to check that from \eqref{Euler13} and \eqref{vorty1} that
\begin{equation*}
\partial_t  \omega^{\eps} +  \div ((v^{\eps} - \ell^{\varepsilon} -r^{\varepsilon} x^{\perp}) \omega^{\eps}) =0 \text{ in } {\mathcal D}'((0,T) \times \R^{2}).
\end{equation*}
Going back to the original variables we infer
\begin{equation} \label{vorty1Dprime}
\partial_t  w^{\eps} +  \div (u^{\eps} w^{\eps}) =0 \text{ in } {\mathcal D}'((0,T) \times \R^{2}).
\end{equation}
In particular, $\partial_t  w^{\eps}$ is bounded in $L^{\infty}(0,T;W^{-1,p}(\R^{2}))$.
Hence we deduce by \cite[Appendix C]{lions} that the convergence \eqref{CVTildeW} can be improved into
\begin{equation} \label{CVTildeW2}
{w}^{\varepsilon_{n}} \longrightarrow {w} \ \text{ in } C^{0}([0,T];L^{p}(\R^2)-w) \ \text{(resp. in } C^{0}([0,T];  L^{\infty}(\R^2)-w*) \text{ if } p=+\infty) \ \text{ as } n \rightarrow +\infty \ .
\end{equation}
Actually, \cite[Appendix C]{lions} considers only the case $p<+ \infty$ since it proves the compactness of a sequence in $C^{0}([0,T];X-w)$ for $X$ a reflexive separable Banach space. However, the generalization to $C^{0}([0,T];  L^{\infty}(\R^2)-w*)$ is straightforward using the separability of $L^{1}(\R^{2})$. \par
Now using Proposition \ref{propdefKwhole} and the Ascoli-Arzel\`a theorem, we see that $K_{\R^{2}}$ is a compact operator form $L^{p}(\R^{2})$ to $L^{\infty}_{loc}(\R^{2})$, so one deduces that
\begin{equation} \label{CVForteVitesse}
K_{\R^{2}} [{w}^{\varepsilon_{n}}] \longrightarrow K_{\R^{2}} [{w}] \ \text{ in } C^{0}([0,T]; L^{\infty}_{loc}(\R^{2}))  \ \text{ as } n \rightarrow +\infty.
\end{equation}
\subsection{Characterization of the limit of the fluid velocity}
{\bf 1.} {\it Convergence of $u ^{\varepsilon} $.} \par
\ \par
\noindent
$\bullet$ We extend  $u ^{\varepsilon} $ by $(h^{\varepsilon})' + r^{\varepsilon} (x - h^{\varepsilon})^{\perp} $ inside ${\mathcal S}^{\varepsilon} (t)$. %
We define $\tilde{u}^{\eps}$ by the relation
\begin{equation}  
\label{allo4}
\tilde{v}^{\eps} (t,x)=Q^{\eps} (t)^T\   \tilde{u}^\eps(t,Q^{\eps}(t)x+h^{\eps}(t)) ,
\end{equation}
so that 
\begin{equation}
\label{allo3}
\tilde{u}^{\eps}: = u^{\eps}- \gamma Q^{\varepsilon} H^{\varepsilon}((Q^{\varepsilon})^{T} (x-h^{\varepsilon}(t))) ,
\end{equation}
%
%
%
%
Using \eqref{phi-scaling} and \eqref{phi-scaling2} we deduce that
\begin{equation*}
\nabla \Phi_{i}^{\varepsilon} \longrightarrow 0 \ \text{ for } i=1,2 \text{ and }
\frac{1}{\varepsilon}\nabla \Phi_{3}^{\varepsilon} \longrightarrow 0 \ \text{ in } L^{2}(\R^{2}) \text{ as } \varepsilon \rightarrow 0^{+}.
\end{equation*}
Here we extended $\nabla \Phi_{i}^{\varepsilon}$ inside ${\mathcal S}_{0}^{\varepsilon}$ by the basis vector $e_i$ for $i=1$ or $2$, and by $x^\perp$ for $i=3$. Consequently, from Proposition \ref{Prop:APE} and \eqref{DefCheckV} we deduce that
\begin{equation} \label{VTVC}
\check{v}^{\varepsilon} - K_{\R^{2}}[\omega^{\varepsilon}] \longrightarrow 0 \ \text{ in } L^{\infty}(0,T;L^{2}(\R^{2})) \text{ as } \varepsilon \rightarrow 0^{+}.
\end{equation}
Gathering \eqref{EstCheckvTildev2} and \eqref{VTVC} we obtain that 
\begin{equation*} 
 K_{\R^{2}}[\omega^{\varepsilon}]  - \tilde{v}^{\eps} \longrightarrow 0 \ \text{ in } L^{\infty}(0,T;L^{2}(\R^{2})) \text{ as } \varepsilon \rightarrow 0^{+}.
\end{equation*}
Using \eqref{allo4}, \eqref{chgtvar2} and the fact that the Biot-Savart law in the plane commutes with translations and rotations, we infer
\begin{equation} \label{VTVC2}
K_{\R^{2}}[w^{\eps}]  - \tilde{u}^{\varepsilon} \longrightarrow 0 \ \text{ in } L^{\infty}(0,T;L^{2}(\R^{2})) \text{ as } \varepsilon \rightarrow 0^{+} .
\end{equation}
 \par
\ \par
\noindent
$\bullet$ Now we use the fact that, since $p'<2$, 
\begin{equation*}
H^{\varepsilon} \longrightarrow H \ \text{ in } L^{p'}_{loc}(\R^{2}) \  \text{ as } \varepsilon \rightarrow 0^{+},
\end{equation*}
where $H$ is given by \eqref{defH}, and where as usual we extend $H^{\varepsilon}$ by $0$ inside ${\mathcal F}^{\varepsilon}_{0}$, see for instance \cite[Lemma 3.11]{Cricri}. Since $H$ is invariant by rotation, it follows from an easy change of variable that 
\begin{equation} \label{CVH}
Q^{\varepsilon_{n}} H^{\varepsilon_{n}}((Q^{\varepsilon_{n}})^{T} (\cdot-h^{\varepsilon_{n}}(t))) \longrightarrow H(\cdot - h(t)) \ \text{ in } L^{p'}_{loc}(\R^{2}) \ \text{ as } n \rightarrow +\infty,
\end{equation}
no matter the rotation matrix $Q^{\varepsilon_{n}}$. \par
%
%
%
%
%
\ \par
\noindent
$\bullet$ Using \eqref{CVForteVitesse}, \eqref{VTVC2} and \eqref{CVH}, we finally deduce that
\begin{equation}
 \label{WS}
{u}^{\varepsilon_{n}}(x) \longrightarrow K_{\R^{2}} [{w}] +\gamma H(\cdot - h(t)) \ \text{ in } L^{\infty}(0,T;L^{p'}_{loc}(\R^{2})) \ \text{ as } n \rightarrow +\infty.
\end{equation}
\ \par
\noindent
{\bf 2.} {\it Fluid equation in the limit.}
\ \par
\noindent
Let us show that  $u$ and $w$  satisfy  \  \eqref{EqSolFaibleIntro}.
Since ${u}^{\varepsilon}$ and ${w}^{\varepsilon}$ satisfy \eqref{vorty1Dprime}, it can be easily seen that for any test function $\psi\in C^\infty_c([0,T)\times\R^2)$,
\begin{equation} \label{SolFEps}
\int_0^\infty\int_{\R^2} \psi_t  {w}^{\varepsilon_{n}} \, dx\, dt
+ \int_0^\infty \int_{\R^2} \nabla \psi \cdot {u}^{\varepsilon_{n}}  {w}^{\varepsilon_{n}} \, dx\, dt 
+ \int_{\R^2} \psi(0,x)  w_0(x) \, dx =0.
\end{equation}
 The convergence as $n \rightarrow +\infty$ of the first term of \eqref{SolFEps} is a direct consequence of \eqref{CVTildeW}. For what concerns the second one, it is a matter of weak/strong convergence since ${u}^{\varepsilon_{n}}$ converges strongly in $L^{\infty}(0,T;L^{p'}_{loc}(\R^{2}))$ (according to \eqref{CVH} and \eqref{WS}) while for ${w}^{\varepsilon}$ we have \eqref{CVTildeW}. 
\subsection{Characterization of the limit of the solid velocity}
We will use the following lemmata, proven in the appendix.
\begin{Lemma}\label{LemCVSolide}
Let $(\rho_{n})_{n \in \N} \in W^{1,\infty}(0,T)^{\N}$ and $(\varepsilon_{n})_{n \in \N} \in (\R_{*}^{+})^{\N}$ such that $(\varepsilon_{n} \rho_{n})$ is bounded in $L^{\infty}(0,T)$ and
\begin{equation} 
\label{EpsilonVers0}
\varepsilon_{n} \longrightarrow 0 \ \text{ as } \  n \rightarrow +\infty.
\end{equation}
Let 
\begin{equation} \label{DefAE}
\alpha_{n}(t):= \int_{0}^{t} \rho_{n}.
\end{equation}
Then
\begin{equation} \label{CV1}
\varepsilon_{n} \rho_{n} \exp(i \alpha_{n}) \stackrel{w*}{\longrightharpoonup} 0 \ \text{ in } \  L^{\infty}(0,T) \ \text{ as } \  n \rightarrow +\infty.
\end{equation}
\end{Lemma}
\begin{Lemma}\label{LemCVSolide2}
Let $(\rho_{n})_{n \in \N} \in W^{1,\infty}(0,T)^{\N}$ and $(\varepsilon_{n})_{n \in \N} \in (\R_{*}^{+})^{\N}$ satisfying \eqref{EpsilonVers0}, and let $\alpha_{n}$ be defined by \eqref{DefAE}. Let $(w_{n})_{n \in \N} \in L^{\infty}(0,T)^{\N}$ such that
\begin{equation} \label{cvwnw}
w_{n} \longrightarrow w \ \text{ in } \  L^{\infty}(0,T) \ \text{ as } \  n \rightarrow +\infty,
\end{equation}
and suppose that
\begin{equation}%
\label{CVEpsilonRho}
\varepsilon_{n} \rho_{n} \stackrel{w*}{\longrightharpoonup} \overline{\rho} \ \text{ in } \  W^{1,\infty}(0,T) \ \text{ as } \  n \rightarrow +\infty,
\end{equation}
and
\begin{equation} \label{condrhow}
\varepsilon_{n} \rho_{n}'(t) = \Re [w_{n}(t) \exp(- i \alpha_{n}(t))].
\end{equation}
Then $\overline{\rho}$ is constant on $[0,T]$.
\end{Lemma}
Now let us establish the behavior of the solid in the limit with the help of these lemmas. First,   using \eqref{CVh}, \eqref{CVForteVitesse} and the uniform estimates on $K_{\R^2}[w^\eps]$ in $C^{1-2/p}(\R^{2})$ if $p< +\infty$ (resp. ${\mathcal{LL}}(\R^{2})$ if $p=+\infty$) given by Proposition \ref{Prop:APE}, we deduce
\begin{equation} \label{CVDiffVit}
(h^{\varepsilon_{n}})' - K_{\R^2}[w^\eps](t,h^{\varepsilon_{n}}) \longrightarrow
h' - K_{\R^2}[{w}](t,h) \ \text{ in } \ L^{\infty}(0,T) \ \text{ as } \ n \rightarrow +\infty.
\end{equation}
Now, we rephrase \eqref{rprime2} with the complex variable:
\begin{equation*}
{\mathcal J}_{0} (\varepsilon_{n} r^{\varepsilon_{n}})' =
 \Re \Big\{ \gamma \overline{\zeta} \exp(-i \theta^{\varepsilon_{n}}) \Big[ ((h^{\varepsilon_{n}})' - K_{\R^2}[w^\eps](t,h^{\varepsilon_{n}}))_{1} + i((h^{\varepsilon_{n}})' - K_{\R^2}[w^\eps](t,h^{\varepsilon_{n}}))_{2} \Big] \Big\} + \tilde{R}_{2}^{\varepsilon_{n}}.
\end{equation*}
We apply Lemma \ref{LemCVSolide2} with $\rho_{n}=r^{\varepsilon_{n}}$, $\alpha_{n}=\theta^{\varepsilon_{n}}$ and
\begin{equation*}
w_{n} := \gamma \overline{\zeta} \Big[ ((h^{\varepsilon_{n}})' - K_{\R^2}[w^\eps](t,h^{\varepsilon_{n}}))_{1} + i((h^{\varepsilon_{n}})' - K_{\R^2}[w^\eps](t,h^{\varepsilon_{n}}))_{2} \Big]  + \exp( i \theta^{\varepsilon_{n}}) \tilde{R}_{2}^{\varepsilon_{n}} .
\end{equation*}
The assumption on $w_{n}$ comes directly from \eqref{CVRTilde} and \eqref{CVDiffVit}. We deduce that the function $R$ defined in \eqref{CVr} is constant. Taking into account that the initial data $r_0$ is independent of $\eps$ we therefore deduce that $\varepsilon_{n} \theta^{\varepsilon_{n}}$ converges to $0$ weakly-$*$ in $W^{2,\infty} (0,T;\R^{2})$. \par
We apply Lemma \ref{LemCVSolide} on \eqref{lprime2} to get rid of the second term in the right hand side and we arrive to \eqref{PointEuler}.
\begin{Remark}
Actually, we do not need to apply Lemma \ref{LemCVSolide} since we know that $\varepsilon r^{\varepsilon}$ converges to $0$ weakly-$*$ in $W^{1,\infty} (0,T)$. It can be noted however that using this lemma, Theorem \ref{MR} can be extended in a straightforward manner to the situation where $r_{0}$ depends on $\varepsilon$ as follows:
\begin{equation*}
\varepsilon r_{0}^{\varepsilon} \longrightarrow R_{0} \text{ as } \varepsilon \rightarrow 0^{+}.
\end{equation*}
In that case, one deduces that $\varepsilon_{n} r^{\varepsilon_{n}}$ converges to $R_{0}$ weakly-$*$ in $W^{1,\infty}(0,T)$, but due to Lemma \ref{LemCVSolide}, no additional term appears in \eqref{PointEuler}.
\end{Remark}
\section{Technical results}
\label{TR}
\subsection{Proof of Proposition \ref{PropKirchoffSueur}}
This is proven in \cite{GS}; we recall it here for the sake of completeness. As we consider $\varepsilon$ fixed here, we omit the $\varepsilon$ in the notations. In particular, here $H$ stands for $H^{\varepsilon}$ and $G$ for $G^{\varepsilon}$. \par
\ \par
\noindent
{\bf 1.} We first give another form of the above Hamiltonian. Let us prove that 
\begin{equation} \label{Hamiltonien2emeForme}
2 \mathcal{H} = m  |\ell(t)|^2 +   \mathcal{J} r(t)^2 +  \int_{ \mathcal{F}_{0} } ( |\hat{v}(t,\cdot )|^2 + 2  (\gamma + \alpha ) \hat{v}(t,\cdot ) \cdot    H   ) \, dx,
\end{equation}
where $\alpha$ is given by \eqref{DefAlpha} and
\begin{equation} \label{hatv}
\hat{v}:= v - (\gamma + \alpha) H.
\end{equation}
Note in particular that
\begin{equation} \label{HatvO1surx2}
\hat{v}(x)= {\mathcal O}(1/|x|^{2}) \ \text{ as }  \ |x| \rightarrow +\infty.
\end{equation}
Let us denote  
\begin{equation*}
\Psi (x)  :=  \int_{\mathcal{F}_{0}} G(x,y) \om(y) dy
\end{equation*}
which is a stream function of  $K[\om]$ vanishing on the boundary $\mathcal{S}_{0}$:
\begin{equation*}
K[\omega]= \nabla^{\perp} \Psi.
\end{equation*}
Let us also  denote 
\begin{equation*}
\nabla \Phi :=  \ell_1 \nabla \Phi_1 +  \ell_2 \nabla \Phi_2 +  r  \nabla \Phi_3,
\end{equation*}
so that
\begin{equation} \label{DCV}
\hat{v} = K [\omega] +  \nabla \Phi.
\end{equation}
Then we compute 
\begin{eqnarray*}
\int_{\mathcal{F}_0}  |\hat{v}|^2 dx = 
\int_{\mathcal{F}_0} \nabla^\perp  \Psi \cdot   \hat{v}
+  \int_{\mathcal{F}_0}  \nabla^\perp   \Psi \cdot  \nabla \Phi 
+ \int_{\mathcal{F}_0}  \nabla \Phi  \cdot  \nabla \Phi  .
\end{eqnarray*}
First, integrating by parts yields
\begin{gather*}
\int_{\mathcal{F}_0} \nabla^\perp \Psi \cdot \hat{v}
= - \int_{\mathcal{F}_0 \times \mathcal{F}_0}  G(x,y) \omega (x) \omega (y) \, dx \, dy , \\
\int_{\mathcal{F}_0}  \nabla^\perp  \Psi \cdot  \nabla \Phi = 0, \\
\int_{\mathcal{F}_0}  \hat{v}  \cdot  H = - \int_{\mathcal{F}_0  } \omega (x)  \Psi_H (x) dx .
\end{gather*}
There is no boundary terms since  $ \Psi$ and $ \Psi_H$ vanish on the boundary $\mathcal{S}_{0}$, and $ \nabla \Phi $ and $ \hat{v}$  decrease also   like $1/ | x |^2$  at infinity. 
\par
Also, by definition, we have 
\begin{eqnarray*}
\int_{\mathcal{F}_0}  \nabla \Phi  \cdot  \nabla \Phi = X^T \mathcal{M}_2 X .
\end{eqnarray*}
This proves \eqref{Hamiltonien2emeForme} by using \eqref{GepsHydro}. \par
\ \par
\noindent
{\bf 2.} We use \eqref{Euler11}, the fact that $\partial_{t} v = \partial_{t} \hat{v} \in L^{\infty}(0,T;L^{q}({\mathcal F}_{0}))$ for any $q$ in $(1,p]$ (resp. in $(1,+\infty)$ if $p=+\infty$) and we notice that $H $ and $v$ are in $ L^{\infty}(0,T;L^{p}({\mathcal F}_{0}))$; this allows to write
\begin{align*}
\mathcal{H}{'} (t) &=   m \ell \cdot \ell{'} (t) + {\mathcal J} r {r}' (t) 
+  \int_{\mathcal{F}_0  } (\partial_ t \hat{v} \cdot \hat{v} + (\gamma + \alpha )  \partial_ t  \hat{v} \cdot H ),  \\ 
&=  m \ell  \cdot \ell' (t) + {\mathcal J} r r' (t) +    \int_{\mathcal{F}_0  } \partial_ t  v \cdot  v,  \\ 
& =  m \ell  \cdot  \ell' (t) + {\mathcal J} r r' (t) - \int_{\mathcal{F}_0  } ( \left[(v-\ell-rx^\perp) \cdot \nabla \right] v 
+ rv^\perp + \nabla q )\cdot  v.
\end{align*}
Then
\begin{equation*}
\mathcal{H}{'} (t) =  I_1 + I_2 + I_3,
\end{equation*}
where
\begin{eqnarray*}
I_1 :=  m \ell   \cdot \ell' (t) + {\mathcal J} r r' (t)  - \int_{\mathcal{F}_0  } \nabla q \cdot  v, \quad
I_2 := -\int_{\mathcal{F}_0  } (v-\ell)\cdot\nabla v  \cdot v , \quad
I_3 := - r  \int_{\mathcal{F}_0} [v^{\perp} - (x^{\perp} \cdot \nabla) v] \cdot v .
\end{eqnarray*}
Let us justify that each integral above is convergent. For $I_2$ the integrability is clear.
For $I_{3}$, we write $I_3 = I_{4} + I_{5} + I_{6}$, with
\begin{equation*}
I_{4}:= -r \int_{\mathcal{F}_0} \hat{v}^\perp     \cdot v , \quad 
I_5 :=  r \int_{\mathcal{F}_0} x^\perp\cdot\nabla  \hat{v} \cdot v , \quad 
I_6:=  (\gamma+\alpha) r \int_{\mathcal{F}_0} (x^\perp\cdot\nabla H - H^\perp )   \cdot v .
\end{equation*}
The integrability of $I_{4}$ is clear, for $I_5$ we use that $\hat{v}$  decreases like $1/ | x |^2$  at infinity, so that $x^\perp\cdot\nabla  \hat{v}$ is integrable and for $I_6$ we use  \eqref{HSeriesLaurent2}. 
It remains to justify the integral in $I_1$.  We analyze the decay of the pressure at infinity, observing that (\ref{Euler11}) reads as follows:
\begin{equation}
\label{Decroiq}
- \nabla q = \partial_t \hat{v}  +\left(v-\ell \right) \cdot \nabla v - r x^\perp \cdot\nabla \hat{v}
+r \hat{v}^\perp + (\alpha+\gamma) r [H^\perp - (x^\perp \cdot\nabla) H],
\end{equation}
so that $\nabla q$  decreases like $1/ | x |^2$  at infinity. Integrating along rays yields that the pressure decreases like $1/ | x |$  at infinity. 
This allows to integrate by parts both  $I_1$ and  $I_2$. \par
\ \par
Now using that \eqref{Euler13} and then the equations (\ref{Solide11})--(\ref{Solide12}), we obtain that  $I_1 = 0$. 
For what concerns $I_{2}$ we get that 
\begin{equation*}
I_2 =  - \frac{1}{2} \int_{\partial \mathcal{S}_0} |v|^2 (v - \ell)\cdot n  .
\end{equation*}
For what concerns $I_{3}$, we consider $R>0$ large in order that ${\mathcal S}_{0} \subset B(0,R)$, and consider the same integral as $I_{3}$, over ${\mathcal F}_{0} \cap B(0,R)$. Integrating by parts we obtain
\begin{equation*}
\int_{\mathcal{F}_0 \cap B(0,R)} [v^{\perp} - (x^{\perp} \cdot \nabla) v] \cdot v
= - \int_{\partial \mathcal{S}_0} (x^{\perp} \cdot n) \frac{|v|^{2}}{2} - \int_{S(0,R)} (x^{\perp} \cdot n) \frac{|v|^{2}}{2},
\end{equation*}
where we denote by $n$ also the unit outward normal on the circle $S(0,R)$. Of course $x^{\perp} \cdot n=0$ on $S(0,R)$, so letting $R \rightarrow +\infty$, we end up with
\begin{equation*}
I_{3} =\frac{1}{2} \int_{\partial \mathcal{S}_0} (rx^{\perp} \cdot n)|v|^{2}.
\end{equation*}
Using \eqref{Euler13} we deduce $I_{2}+I_{3}=0$, so in total we get $\mathcal{H}{'} (t) = 0$.
\par
\subsection{Proofs of Lemmas \ref{blasius} and \ref{Blabla}}
\begin{proof}[Proof of Lemma \ref{blasius}]
By polarization, it is sufficient to consider the case where $f=g$. Let us consider $\gamma= (\gamma_{1},\gamma_{2}): [0,1] \rightarrow \R^{2}$ a smooth parameterization of the Jordan curve ${\mathcal C}$. On one side, one has
\begin{equation} \label{cotereel}
\int_{{\mathcal C}} (f \cdot f) n \, ds = \int_{0}^{1} \big( f_{1}(\gamma(t))^{2} + f_{2}(\gamma(t))^{2} \big) \begin{pmatrix} -{\gamma}_{2}'(t) \\ {\gamma}_{1} ' (t)\end{pmatrix} \, dt.
\end{equation}
On the other side, one has
\begin{equation} \label{cotecomplexe}
\int_{{\mathcal C}} (f_{1}(z) - i f_{2}(z))^{2} \, dz
= \int_{{\mathcal C}} \Big( f_{1}(\gamma(t)) - i f_{2}(\gamma(t)) \Big)
\Big[ \big( f_{1}(\gamma(t)) - i f_{2}(\gamma(t)) \big) ({\gamma}_{1}'(t) + i {\gamma}_{2}'(t)) \Big]  \, dt .
\end{equation}
But since $f$ is tangent to ${\mathcal C}$, one sees that the expression inside the brackets in \eqref{cotecomplexe} is real, and hence is equal to its complex conjugate. It follows that
\begin{equation*}
\int_{{\mathcal C}} (f_{1}(z) - i f_{2}(z))^{2} \, dz
= \int_{{\mathcal C}} \big| f_{1}(\gamma(t)) - i f_{2}(\gamma(t)) \big|^{2} ({\gamma}_{1}'(t) - i {\gamma}_{2}'(t))  \, dt ,
\end{equation*}
and \eqref{bla1} follows. \par
The proof of \eqref{bla2} is analogous: using again 
\begin{equation*}
(f_{1} (\gamma(t)) - i f_{2} (\gamma(t)))( {\gamma}_{1}'(t) + i {\gamma}_{2}'(t))
=(f_{1} (\gamma(t)) + i f_{2} (\gamma(t)))({\gamma}_{1}'(t) - i {\gamma}_{2}'(t)),
\end{equation*}
we deduce
\begin{equation*}
\int_{{\mathcal C}} (f_{1}(z) - i f_{2}(z))^{2} z \, dz
= \int_{{\mathcal C}} |f_{1}(\gamma(t)) - i f_{2}(\gamma(t))|^{2} (\gamma_{1}(t) + i \gamma_{2}(t)) ({\gamma}_{1}'(t) - i {\gamma}_{2}'(t))  \, dt ,
\end{equation*}
so that
\begin{eqnarray*}
\Re \left(\int_{{\mathcal C}} (f_{1}(z) - i f_{2}(z))^{2} z \, dz\right)
&=& \int_{{\mathcal C}} \Big( f_{1}(\gamma(t))^{2} + f_{2}(\gamma(t))^{2} \Big) \big( \gamma_{1}(t) {\gamma}_{1}'(t) +  \gamma_{2}(t) {\gamma}_{2}'(t) \big)  \, dt \\
&=& \int_{ \mathcal{C}} (f  \cdot f) (x^{\perp} \cdot n)  \,ds .
\end{eqnarray*}
\end{proof}
\begin{proof}[Proof of Lemma \ref{Blabla}]
Parameterizing $\partial  \mathcal{S}_{0}$ by $\gamma= (\gamma_{1},\gamma_{2}): [0,1] \rightarrow \R^{2}$ as previously we have
\begin{align*}
\int_{\partial \mathcal{S}_{0}} \overline{z} (H_{1} -i H_{2})  z \, dz 
&= \int_{\partial \mathcal{S}_{0}} (\gamma_{1}^{2}(t) + \gamma_{2}^{2}(t))
\Big[ ( H_{1}(\gamma(t)) - i H_{2}(\gamma(t)) ) ({\gamma}_{1}'(t) + i {\gamma}_{2}'(t)) \Big]  \, dt .
\end{align*}
Since $H$ is tangent to $\partial {\mathcal S}_{0}$, the imaginary part of the bracket above is zero, so the integral is real.
\end{proof}
\subsection{Proofs of Lemmas \ref{LemCVSolide} and \ref{LemCVSolide2}}
\begin{proof}[Proof of Lemma \ref{LemCVSolide}]
We first note that the sequence $(\varepsilon_{n} \rho_{n} \exp(i \alpha_{n}))_{n \in \N}$ is bounded in $L^{\infty}(0,T)$, so it suffices to prove that the convergence \eqref{CV1} takes place in the sense of distributions. Given $\varphi \in C^{\infty}_{0}((0,T))$, we see that
\begin{eqnarray*}
\int_{a}^{b} \varepsilon_{n} \rho_{n}(t) \exp(i \alpha_{n}(t)) \varphi(t) \, dt 
&=& -i \varepsilon_{n} \int_{a}^{b} i \alpha_{n}'(t) \exp(i \alpha_{n}(t)) \varphi(t) \, dt  \\
&=& i \varepsilon_{n} \int_{a}^{b} \exp(i \alpha_{n}(t)) \varphi'(t) \, dt,
\end{eqnarray*}
which yields the desired convergence. 
\end{proof}
\ \par
\begin{proof}[Proof of Lemma \ref{LemCVSolide2}]
Let us define the following open subset of $(0,T)$:
\begin{equation} \label{EnsembleA}
A:= \{ x \in (0,T) \ / \ \overline{\rho}(x) \not =0 \}.
\end{equation}
\ \par
\noindent
$\bullet$  It is classical that, by the Lipschitz character of the function $\overline{\rho}$ only, one has
\begin{equation*}
\overline{\rho}'=0 \ \text{ a. e. on } \  (0,T) \setminus A. 
\end{equation*}
\ \par
\noindent
$\bullet$ Now let us consider what happens on $A$. Define
\begin{equation*}
\Theta(t):= \int_{0}^{t} \overline{\rho}.
\end{equation*}
Consider a connected component of $A$, say $(a,b)$. Let us show that, due to the nonstationary phase $\Theta' \not =0$ in $(a,b)$, one has 
\begin{equation} \label{CVDalpha}
\exp(- i \alpha_{n}) \stackrel{w*}{\longrightharpoonup} 0 \ \text{ in } \ L^{\infty}(a,b) \ \text{ as } \ n \rightarrow +\infty.
\end{equation}
Of course, it is sufficient to prove this convergence in the sense of ${\mathcal D}'(a,b)$. Hence, let us consider $\varphi \in C^{\infty}_{0}((a,b))$, say $\mbox{Supp} (\varphi) \subset [\tilde{a},\tilde{b}] \subset (a,b)$. Note that by \eqref{DefAE} and \eqref{CVEpsilonRho}, one has
\begin{equation*}
\varepsilon_{n} \alpha_{n} \stackrel{w*}{\longrightharpoonup} \Theta \ \text{ in }\  W^{2,\infty}([0,T]).
\end{equation*}
Consequently there exists $\kappa>0$ such that for all $n$ large enough, one has
\begin{equation} \nonumber
\varepsilon_{n} |\alpha_{n}'(t)| \geq \kappa>0 \ \text{ and } \ 
\varepsilon_{n} |\alpha_{n}'(t)| + \varepsilon_{n} |\alpha_{n}''(t)| \leq \kappa^{-1} \ 
 \text{ on } [\tilde{a},\tilde{b}].
\end{equation}
For such $n$ one has
\begin{eqnarray*}
\int_{a}^{b} \exp(-i \alpha_{n}(t)) \varphi(t) \, dt
&=& \int_{\tilde{a}}^{\tilde{b}} \alpha'_{n}(t)\exp(-i \alpha_{n}(t)) \frac{\varphi(t)}{\alpha_{n}'(t)} \, dt \\
&=& -i \int_{\tilde{a}}^{\tilde{b}} \exp(-i \alpha_{n}(t)) \frac{\varphi'(t) \alpha_{n}'(t) - \varphi(t) \alpha_{n}''(t)}{(\alpha'_{n})^{2}(t)} \, dt \\
&=& -i \varepsilon_{n} \int_{\tilde{a}}^{\tilde{b}} \exp(-i \alpha_{n}(t)) \frac{\varphi'(t) (\varepsilon_{n}\alpha_{n})'(t) - \varphi(t) (\varepsilon_{n}\alpha_{n})''(t)}{(\varepsilon_{n}\alpha'_{n})^{2}(t)} \, dt,
\end{eqnarray*}
and \eqref{CVDalpha} follows. \par
Hence, by weak/strong convergence we deduce from \eqref{cvwnw}, \eqref{condrhow} and \eqref{CVDalpha} that for any $\varphi \in C^{\infty}_{0}((a,b);\R)$,
\begin{equation*}
\varepsilon_{n} \langle \rho_{n}', \varphi \rangle_{L^{\infty} \times L^{1}} = \Re \big( \langle \exp(-i \alpha_{n}), w_{n} \varphi \rangle_{L^{\infty} \times L^{1}} \big) \longrightarrow 0 \ \text{ as } \  n \rightarrow +\infty.
\end{equation*}
Consequently, on each connected component $(a,b)$ of $A$, we obtain that
\begin{equation*}
\overline{\rho}'=0 \ \text{ a. e. on } \  (a,b).
\end{equation*}
Of course there is at most a countable quantity of such connected components. Hence we obtain that $\overline{\rho}'=0$ a.e. on $(0,T)$ and the conclusion follows.
\end{proof}
\section{Appendix. Proof of Theorem \ref{ThmYudo}}
\label{PreuveYudo}
We first prove a result  of global in time existence and uniqueness similar to the celebrated result by 
Yudovich  about a fluid alone. We recall that the space $\mathcal{LL}$ was defined in \eqref{DefLL}.
\begin{Theorem} \label{ThmYudo0}
For any $u_0 \in C^{0}(\overline{\mathcal{F}_0};\R^{2})$, $(\ell_0,r_0) \in \R^2 \times \R$, such that:
\begin{equation} \label{CondCompatibiliteYudo}
\div u_0 =0 \text{ in } {\mathcal F}_0 \ \text{ and } \  u_0   \cdot  n = (\ell_0 + r_0 x^{\perp})   \cdot  n \text{ on } \partial \mathcal{S}_0,
\end{equation}
\begin{equation} \label{TourbillonYudoinf}
w_0 := \curl u_0  \in L_c^{\infty}(\overline{{\mathcal F}_0}),
\end{equation}
\begin{equation*}
\lim_{|x| \rightarrow +\infty} u_{0}(x) =0,
\end{equation*}
there exists a unique solution $(h',r,u)$ of \eqref{Euler1}--\eqref{Solideci} in $C^1 (\R^+; \R^2 \times \R) \times L^{\infty}(\R^+, \mathcal{LL}({\mathcal F}(t))$. Moreover for all $t>0$, $w(t):=\curl u(t) \in L^\infty_c(\overline{{\mathcal F}(t)})$.
\end{Theorem}
We first prove the local in time existence part by Schauder's fixed point theorem. The global in time existence follows then from our a priori estimates of Section {\ref{Sec:APE}}. Finally we follow Yudovich's approach for what concerns the uniqueness.  \par
\subsection{Proof of Theorem \ref{ThmYudo0}}
\ \par
\noindent
{\bf Reformulating the problem.} To begin with, we consider the equations in the body frame as in Subsection \ref{Subsec:VEBF}. 
Next, we decompose the pressure. To this purpose we recall the following result from  \cite{SimaderSohr} about the Leray projector on ${\mathcal F}_{0}$:
\begin{Lemma} \label{LemProjLeray}
For any $q\in (1,\infty)$, we have
$$ L^q ({\mathcal F}_{0};\R^{2})  = L^q_\sigma ({\mathcal F}_{0} ) \oplus E^q ({\mathcal F}_{0} ), $$
where 
\begin{gather*}
L^q_\sigma ({\mathcal F}_{0} ) :=  \{ u \in  L^q ({\mathcal F}_{0};\R^{2}) / \ \div u= 0 \text{ in } {\mathcal F}_{0} \text{ and } u \cdot n = 0 \text{ on } \partial {\mathcal S}_{0} \}, \\
E^q ({\mathcal F}_{0} ) := \{  \nabla p \ / \ p \in  L^q_\text{loc} ( \overline{{\mathcal F}_{0}};\R)  \text{ and }  \nabla p  \in L^q({\mathcal F}_{0};\R^{2}) \} .
\end{gather*}
Moreover the projection $P_q$ from  $L^q ({\mathcal F}_{0};\R^{2} ) $ onto $L^q_\sigma ({\mathcal F}_{0} )$ along $E^q ({\mathcal F}_{0})$ is linear continuous.
\end{Lemma}
Since, by density of smooth compactly supported vector fields in $L^q ({\mathcal F}_{0};\R^{2})$, $P_q$ and $P_{\tilde{q}}$ coincide  on $L^q ({\mathcal F}_{0};\R^{2})  \cap L^{\tilde{q}}({\mathcal F}_{0};\R^{2}) $, we will simply denote $P$ without dwelling. \par
This allows to reformulate the solid equation as follows. We introduce $\mu$ up to an additive constant by
\begin{equation} \label{DefNablaMu}
\nabla \mu := - (Id -P ) \Big( (v-\ell-r x^\perp  ) \cdot\nabla v + r v^{\perp} \Big).
\end{equation}
We will see that in the case under view this is well defined in $L^{\infty}_{loc}(\R^{+};L^{p}({\mathcal F}_{0}))$.
Then, assuming that the solution has the regularity claimed in Theorem \ref{ThmYudo}, we observe that 
\begin{equation*}
\partial_{t} v - \begin{bmatrix}  \ell \\ r \end{bmatrix}' \cdot (\nabla \Phi_{i})_{i=1,2,3} \in L^{p}_{\sigma}({\mathcal F}_{0}).
\end{equation*}
We deduce that
\begin{eqnarray*}
\begin{bmatrix}  \ell \\ r \end{bmatrix}'  \cdot (\nabla \Phi_{i})_{i=1,2,3} =  (Id -P ) \partial_{t } v ,
\end{eqnarray*}
where  the functions $\Phi_{i}$ are defined in \eqref{t1.3sec}-\eqref{t1.5sec}. 
We obtain that 
%
the pressure $q$  can be decomposed into 
\begin{equation} \label{CalcPression}
 \nabla   q =  \nabla   \mu- \begin{bmatrix}  \ell \\ r \end{bmatrix}'  \cdot (\nabla \Phi_{i})_{i=1,2,3} .
\end{equation}
It is now elementary to see that the equation of the solid reads
\begin{equation} \label{CalcPression2}
{\mathcal M} \begin{pmatrix} \ell^{1} \\ \ell^{2} \\ r \end{pmatrix}'
=  \Bigg( \int_{{\mathcal F}_{0}} \nabla \mu (\tau,x) \cdot \nabla \Phi_{i}(x) \, dx \Bigg)_{i=1,2,3} 
-m r \begin{pmatrix} -\ell^{2} \\ \ell^{1} \\ 0 \end{pmatrix},
\end{equation}
where ${\mathcal M}$ is given in \eqref{InertieMatrix}. \par
\ \par
\noindent
{\bf An operator.} Let us introduce an operator, whose fixed points give local in time solutions to the system. We denote
\begin{equation*}
\overline{\rho}:= \min \{ \rho >0 \ / \ \mbox{Supp\,}(w_{0}) \subset \overline{B}(0,\rho) \},
\end{equation*}
and as before we denote $\gamma$ the circulation of $u_{0}$ around $\partial {\mathcal S}_{0}$.
For $T>0$, we let
\begin{align*}
{\mathcal C}:= \Big\{ (\omega,\ell,r) \in & \ L^{\infty}(0,T; L^{\infty}_{c}({\mathcal F}_{0})) \times W^{1,1}([0,T];\R^{3}) \ \big/ \\
& {\it i. }\ \  \| \omega \|_{L^{\infty}(0,T;L^{1}({\mathcal F}_{0}))} \leq \| w_{0} \|_{L^{1}({\mathcal F}_{0})}, \  
\| \omega \|_{L^{\infty}(0,T;L^{\infty}({\mathcal F}_{0}))} \leq \| w_{0} \|_{L^{\infty}({\mathcal F}_{0})}, \\
& \hskip 1cm \text{ and } \int_{{\mathcal F}_{0}} w(t,x) \, dx = \int_{{\mathcal F}_{0}} w_{0}(x) \, dx, \ \forall t \in [0,T], \\ 
&  {\it ii. }\ \  \mbox{Supp\,}(\omega(t)) \subset \overline{B}(0,\overline{\rho} + 1), \\ 
& {\it iii. }\ \  \partial_{t} \omega \in L^{1}(0,T;W^{-1,3}({\mathcal F}_{0}))
\ \text{ and } \ \  \| \partial_{t} \omega \|_{L^{1}(0,T;W^{-1,3}({\mathcal F}_{0}))} \leq 1, \\
&  {\it iv. }\ \  \| \ell - \ell_{0} \|_{W^{1,1}(0,T)} \leq 1, \  \| r - r_{0} \|_{W^{1,1}(0,T)} \leq 1
\Big\}.
\end{align*}
Now we define the operator ${\mathcal V}:{\mathcal C} \rightarrow {\mathcal C}$, mapping $(\omega,\ell,r)$ to $(\tilde{\omega},\tilde{\ell},\tilde{r})$ as follows. \par
We first introduce $v$ as the solution of \eqref{DivCurlSystem} associated to $(\omega,\ell,r)$ and $\gamma$. As is classical (see Proposition \ref{propdefK}), the resulting $v$ is bounded and log-Lipschitz uniformly in time with
\begin{equation} \label{EstLL}
\| v \|_{L^{\infty}(0,T;\mathcal{LL}({\mathcal F}_{0}))} \leq C({\mathcal S}_{0}) ( \| \omega \|_{L^{\infty}(0,T;L^{1}({\mathcal F}_{0}))} + \| \omega \|_{L^{\infty}(0,T;L^{\infty}({\mathcal F}_{0}))} + \| \ell \|_{L^{\infty}(0,T)} + \| r \|_{L^{\infty}(0,T)} + |\gamma|),
\end{equation}
and it also satisfies (see e.g. \cite{GT}) for $p \in (1,+\infty)$:
\begin{equation} \label{EstW1p}
\| \nabla v \|_{L^{\infty}(0,T;L^{p}({\mathcal F}_{0}))} \leq C({\mathcal S}_{0})  \,  \frac{p^2}{p-1} ( \| \omega \|_{L^{\infty}(0,T;L^{1}({\mathcal F}_{0}))} + \| \omega \|_{L^{\infty}(0,T;L^{\infty}({\mathcal F}_{0}))} + \| \ell \|_{L^{\infty}(0,T)} + \| r \|_{L^{\infty}(0,T)} + |\gamma|).
\end{equation}
As a consequence of \eqref{EstLL}, we have also a uniform log-Lipschitz estimate on $v - \ell - rx^{\perp}$, so one can define a unique flow associated to $v - \ell - rx^{\perp}$, that is $\Phi$ such that
\begin{equation*}
\partial_{t} \Phi(t,s,x) = (v - \ell - rx^{\perp})(t,\Phi(t,s,x)) \ \text{ and }\ \Phi(s,s,x)=x \ \text{ for } (t,s,x) \in [0,T]^{2} \times {\mathcal F}_{0}.
\end{equation*}
Then the $\omega$-part of ${\mathcal V}$ is defined as
\begin{equation*}
\tilde{\omega}(t,x):= w_{0}(\Phi(0,t,x)).
\end{equation*}
It satisfies
\begin{equation} \label{EqTildeOmega}
\partial_{t} \tilde{\omega} + \div ((v-\ell-rx^{\perp}) \tilde{\omega})=0.
\end{equation}
Next we introduce $\mu$ as
\begin{equation} \label{SysMu}
\nabla \mu := - (Id -P ) \Big ( (v-\ell-r x^\perp  )\cdot \nabla v + r v^{\perp} \Big).
\end{equation}
Let us justify that $\nabla \mu$ is well defined in $L^{\infty}(0,T;L^{p}({\mathcal F}_{0}))$, $p>2$. Due to Lemma \ref{LemProjLeray} one has only to check that $(v-\ell-r x^\perp  )\cdot \nabla v + r v^{\perp} \in L^{\infty}(0,T;L^{p}({\mathcal F}_{0}))$. For what concerns the part $(v-\ell)\cdot \nabla v$, this comes directly from \eqref{EstLL} and \eqref{EstW1p}. For what concerns $r (x^\perp \cdot \nabla) v$ and  $r v^{\perp} $ we use the fact that uniformly in $t$,
\begin{equation} \label{valinfini}
{v}={\mathcal O}\left(\frac{1}{|x|}\right) \text{ and } \nabla v = {\mathcal O}\left(\frac{1}{|x|^{2}}\right) \ \text{ as } x \rightarrow +\infty.
\end{equation}
To prove \eqref{valinfini}, we recall that $v$ is harmonic for $|x|$ large and tends to zero at infinity, and use the following estimate for its (convergent) Laurent series development: for $R>0$ large enough, one has
\begin{equation*}
\Big\| z \sum_{k\geq 1} \frac{a_{k}}{z^{k}} \Big\|_{L^{\infty}(\C \setminus B(0,R))} \leq 
\max \left( |a_{1}|, \Big\| z \sum_{k\geq 1} \frac{a_{k}}{z^{k}} \Big\|_{L^{\infty}(S(0,R))} \right)
\leq C \Big\| \sum_{k\geq 1} \frac{a_{k}}{z^{k}} \Big\|_{L^{\infty}(S(0,R))},
\end{equation*}
as follows by the maximum principle and Cauchy's Residue Theorem. The conclusion on $v$ follows then from the $L^{\infty}({\mathcal F}_{0})$  part of the  estimate \eqref{EstLL}. We proceed analogously on $\nabla v$, adding the classical interior elliptic regularity estimate: 
\begin{equation*}
\Big\| \sum_{k\geq 1} \frac{b_{k}}{z^{k}} \Big\|_{L^{\infty}(S(0,R))} \leq 
C \Big\| \sum_{k\geq 1} \frac{b_{k}}{z^{k}} \Big\|_{L^{p}(B(0,R+1) \setminus B(0,R-1))},
\end{equation*}
and using \eqref{EstW1p}. The estimate \eqref{valinfini} follows. \par
\ \par
Now we can define $(\tilde{\ell},\tilde{r})$ as follows:
\begin{equation} \label{TildeLR}
\begin{pmatrix} \tilde{\ell}^{1} \\ \tilde{\ell}^{2} \\ \tilde{r} \end{pmatrix}(t)
= \begin{pmatrix} \ell^{1}_{0} \\ \ell^{2}_{0} \\ r_{0} \end{pmatrix} + {\mathcal M}^{-1}\int_{0}^{t} 
\Bigg[ \begin{pmatrix} \displaystyle \int_{{\mathcal F}_{0}} \nabla \mu (\tau,x) \cdot \nabla \Phi_{i}(x) \, dx \end{pmatrix}_{i=1,2,3}
-m r \begin{pmatrix} -\ell^{2} \\ \ell^{1} \\ 0 \end{pmatrix} \Bigg] \, d \tau.
\end{equation}
\ \par
\noindent
{\bf Existence of a solution via a fixed point.} Let us now prove that for $T>0$ suitably small, ${\mathcal V}$ admits a fixed point in ${\mathcal C}$. We endow ${\mathcal C}$ with the $L^{\infty}(0,T;L^{p}({\mathcal F}_{0})) \times C^{0}([0,T];\R^{3})$ topology (for some $p \in (2,+\infty)$) and use Schauder's fixed point theorem. It is clear that ${\mathcal C}$ is a closed convex subset of $L^{\infty}(0,T;L^{p}({\mathcal F}_{0}))  \times C^{0}([0,T];\R^{3})$. Hence it remains to prove that ${\mathcal V}({\mathcal C}) \subset {\mathcal C}$, that ${\mathcal V}({\mathcal C})$ is relatively compact and that ${\mathcal V}$ is continuous. \par
\ \par
\noindent
$\bullet$ Let $(\omega,\ell,r) \in {\mathcal C}$. That $\tilde{\omega}$ satisfies point ${\it i.}$ in the definition of ${\mathcal C}$ is immediate. Due to \eqref{EstLL}, we see that
\begin{equation*}
\| v - \ell \|_{\infty} \leq C({\mathcal S}_{0}) ( \| w_{0} \|_{1} + \| w_{0} \|_{\infty} + | \ell_{0} | + | r_{0} | + |\gamma| +1),
\end{equation*}
so that $\mbox{Supp\,}(\tilde{\omega}(t)) \subset \overline{B}(0,\overline{\rho} + 1)$ is granted for $T$ small. 
For what concerns point ${\it iii.}$ in the definition of ${\mathcal C}$, we simply use  \eqref{EqTildeOmega}
and the estimates on $v$, $\ell$, $r$ and $\mbox{Supp\,}(\tilde{\omega})$ to see that it is satisfied by $\tilde{\omega}$ for $T$ suitably small. \par
Due to \eqref{EstLL}, \eqref{EstW1p}, \eqref{SysMu} and Lemma \ref{LemProjLeray}, we deduce that 
\begin{equation} \label{EstNablaMu}
\| \nabla \mu \|_{L^{p}({\mathcal F}_{0})} \leq C({\mathcal S}_{0}) ( \| w_{0} \|_{1} + \| w_{0} \|_{\infty} + | \ell_{0} | + | r_{0} | + |\gamma| +1)^{2}.
\end{equation}
Hence with \eqref{TildeLR} we obtain easily that for $T$ suitably small, $(\tilde{\ell},\tilde{r})$ satisfies the estimates in the definition of ${\mathcal C}$. Hence we have ${\mathcal V}({\mathcal C}) \subset {\mathcal C}$ for $T$ small. \par
\ \par
\noindent
$\bullet$ Let us now prove that ${\mathcal V}({\mathcal C})$ is relatively compact. Let us consider $(\tilde{\omega}_{n},\tilde{\ell}_{n},\tilde{r}_{n})$ a sequence in ${\mathcal V}({\mathcal C})$, let us say $(\tilde{\omega}_{n},\tilde{\ell}_{n},\tilde{r}_{n}) = {\mathcal V}(\omega_{n},\ell_{n},r_{n})$. 
Call $v_{n}$ the velocity field associated to $(\omega_{n},\ell_{n},r_{n})$ by \eqref{DivCurlSystem}, and $\Phi_{n}$ the corresponding flow. Using the definition of ${\mathcal C}$, \eqref{EstW1p} and Aubin-Lions' lemma, one deduces that $(v_{n})$ is relatively compact in $L^{\infty}_{loc}([0,T] \times \overline{{\mathcal F}_{0}})$. Let us say that $v_{\varphi(n)}$ converges uniformly to $v$. Due to the uniform log-Lipschitz estimates on $v$, we infer that
$\Phi_{\varphi(n)}$ converges uniformly towards the flow $\Phi$ associated to $v$ on $[0,T]^{2} \times \overline{B}(0,\overline{\rho}+1)$. This involves that
\begin{equation*}
\tilde{\omega}_{n} \longrightarrow \tilde{\omega}(t,x):=w_{0}(\Phi(0,t,x)) \ \text{ in } L^{\infty}(0,T;L^{p}({\mathcal F}_{0})).
\end{equation*}
(This convergence can for instance be established by using the density of $C^{\infty}_{0}({\mathcal F}_{0})$ in $L^{p}({\mathcal F}_{0})$.)
The compactness for $(\tilde{\ell},\tilde{r})$ is straightforward. \par
\ \par
\noindent
$\bullet$ We finally prove the continuity of ${\mathcal V}$. Suppose that $(\omega_{n},\ell_{n},r_{n})$ converges to $(\omega,\ell,r)$ in $L^{\infty}(0,T;L^{p}({\mathcal F}_{0})) \times C^{0}([0,T];\R^{3})$. Associate to them $v_{n}$, $\Phi_{n}$, etc. and $v$, $\Phi$, etc. Then as previously one deduces that $v_{n}$ converges to $v$ uniformly and that $\nabla v_{n}$ converges to $\nabla v$ in $L^{\infty}(0,T;L^{p}({\mathcal F}_{0}))$. We deduce consequently that $\Phi_{n}$ converges to $\Phi$ uniformly, so in the same way as above, $\tilde{\omega}_{n}$ converges to $\tilde{\omega}$ in $L^{\infty}(0,T;L^{p}({\mathcal F}_{0}))$. Also, using the fact that the estimates \eqref{valinfini} are uniform in $t$ and $n$, we see that $\nabla \mu_{n}$ converges to  $\nabla \mu$ in $L^{\infty}(0,T;L^{p}({\mathcal F}_{0}))$. We deduce that $(\ell_{n},r_{n})$ converges uniformly to $(\ell,r)$. \par
\ \par
\noindent
$\bullet$ Hence for $T$ small we obtain a fixed point. 
One deduces that this yields a solution to the original system \eqref{Euler1}-\eqref{Solideci}, going back to the original frame. Then the fact that a maximal solution is global comes from the a priori estimates on $\omega$, and the estimates of Section \ref{Sec:APE} for what concerns $\mbox{Supp\,}(\omega(t))$, $\ell$ and $r$. \par
\ \par
\noindent
{\bf Uniqueness.} This relies on Yudovich's idea \cite{Yudovich}. Suppose that we have two solutions $(\ell_{1},r_{1},v_{1})$ and $(\ell_{2},r_{2},v_{2})$ with the same initial data. (In this part of the proof the indices do not stand for the components.)  In particular, they share the same circulation $\gamma$ and initial vorticity $w_{0}$. As a consequence, despite the fact that $v_{1}$ and $v_{2}$ are not necessarily in $L^{2}({\mathcal F}_{0})$, their difference $v_{1} - v_{2}$ does belong to $L^{\infty}(0,T;L^{2}({\mathcal F}_{0}))$ with
\begin{equation} \label{EstDiffv}
v_{1} - v_{2} = {\mathcal O}\left( \frac{1}{|x|^{2}}\right) \text{ and } \nabla( v_{1} - v_{2}) = {\mathcal O}\left( \frac{1}{|x|^{3}}\right) \ \text{ as } |x| \rightarrow +\infty.
\end{equation}
(Recall that both $v_{1}$ and $v_{2}$ are harmonic for $|x|$ large enough and converge to $0$ at infinity). \par
Let us also remark that due to \eqref{EstNablaMu} and \eqref{CalcPression2}, $\ell_{1}$,  $\ell_{2}$, $r_{1}$, $r_{2}$ belong to $W^{1,\infty}(0,T)$. As a consequence, using \eqref{CalcPression}, we see that $\nabla q_{1}$ and $\nabla q_{2}$ belong to $L^{\infty}(0,T;L^{2}({\mathcal F}_{0}))$. \par
\ \par
Then defining $\breve{\ell}:=\ell_{1}-\ell_{2}$, $\breve{r}:=r_{1}-r_{2}$, $\breve{v}:= v_{1} -v_{2}$ and $\breve{q}=q_{1}-q_{2}$,  we deduce from \eqref{Euler11} that
\begin{equation*}
\frac{\partial \breve{v}}{\partial t}
+ \left[(v_{1} -\ell_{1} -r_{1} x^\perp)\cdot\nabla\right]\breve{v} 
+ \left[(\breve{v} -\breve{\ell} -\breve{r} x^\perp)\cdot\nabla\right] v_{2} 
+ r_{1} \breve{v}^\perp + \breve{r} v_{2}^{\perp}+\nabla \breve{q} =0 .
\end{equation*}
\par
We multiply by $\breve{v}$, integrate over ${\mathcal F}_{0}$ and integrate by parts (which is permitted by \eqref{EstDiffv} and by the regularity of the pressure), and deduce:
\begin{equation*}
\frac{1}{2} \frac{d}{dt} \| \breve{v} \|_{L^{2}}^{2}
+ \int_{{\mathcal F}_{0}}  \breve{v} \cdot \left[ (\breve{v} -\breve{\ell} -\breve{r} x^\perp)\cdot \nabla v_{2} \right]  \, dx
+\breve{r} \int_{{\mathcal F}_{0}}  \breve{v}   \cdot v_{2}^{\perp} \, dx
+ \int_{\partial {\mathcal F}_{0}} \breve{q} \breve{v}\cdot n=0 .
\end{equation*}
For what concerns the last term, 
\begin{eqnarray*}
\int_{\partial {\mathcal F}_{0}} \breve{q} \breve{v}\cdot n &=& \breve{\ell} \cdot \int_{\partial {\mathcal F}_{0}} \breve{q} n + \breve{r}  \int_{\partial {\mathcal F}_{0}} \breve{q} x^{\perp} \cdot n \\
&=& m\breve{\ell} \cdot \big( \breve{\ell}'+ \breve{r}\ell_{1}^{\perp} + r_{2} \breve{\ell}^{\perp} \big) + {\mathcal J} \breve{r}\breve{r}' \\
&=& m \breve{r}\, \breve{\ell} \cdot \ell_{1}^{\perp} + m\breve{\ell} \cdot \breve{\ell}'+ {\mathcal J} \breve{r}\breve{r}'.
\end{eqnarray*}
Using \eqref{vect}, we deduce
\begin{equation*}
(x^\perp \cdot \nabla) v_{2} = \nabla (x^{\perp} \cdot v_{2}) - v_{2}^{\perp} - x^{\perp} \omega_{2},
\end{equation*}
so that after integration by parts
\begin{equation*}
\int_{{\mathcal F}_{0}} \breve{v} \cdot [(x^\perp \cdot \nabla) v_{2}] \, dx 
= \int_{{\mathcal S}_{0}} (x^{\perp} \cdot v_{2}) [(\breve{\ell} + \breve{r} x^{\perp}) \cdot n ]\, ds+
\int_{{\mathcal F}_{0}} \breve{v} \cdot (- v_{2}^{\perp} - x^{\perp} \omega_{2}) \, dx  .
\end{equation*}
Hence using the boundedness of $v_{2}$ and $\omega_{2}$ in $L^{\infty}(0,T;L^{\infty}({\mathcal F}_{0}))$, the boundedness of $\ell^{1}$ and the one of $\mbox{Supp\,}(\omega_{2})$, we arrive to
\begin{equation*}
\frac{d}{dt} \big( \| \breve{v} \|_{L^{2}}^{2} + \| \breve{\ell} \|^{2} + \| \breve{r} \|^{2} \big) 
\leq C  \Big( \| \breve{v} \|_{L^{2}}^{2} + \| \breve{\ell} \|^{2} + \| \breve{r} \|^{2} + \| \nabla v_{2} \|_{L^{p}} \| \breve{v}^{2} \|_{L^{p'}} \Big),
\end{equation*}
for $p>2$.
(Here, the various constants $C$ may depend on ${\mathcal S}_{0}$ and on the solutions $(\ell_{1},r_{1},v_{1})$ and $(\ell_{2},r_{2},v_{2})$, but not on $p$.) Hence using \eqref{EstW1p}, we obtain that for $p$ large,
\begin{eqnarray*}
\frac{d}{dt} \big( \| \breve{v} \|_{L^{2}}^{2} + \| \breve{\ell} \|^{2} + \| \breve{r} \|^{2} \big) 
&\leq& C  \Big( \| \breve{v} \|_{L^{2}}^{2} + \| \breve{\ell} \|^{2} + \| \breve{r} \|^{2} \Big) + \tilde{C} p \| \breve{v}^{2} \|_{L^{p'}} \\
&\leq& C  \Big( \| \breve{v} \|_{L^{2}}^{2} + \| \breve{\ell} \|^{2} + \| \breve{r} \|^{2} \Big) 
+ \tilde{C} p \| \breve{v} \|_{L^{2}}^{\frac{2}{p'}}  \| \breve{v}^{2} \|_{L^{\infty}}^{\frac{1}{p}}  .
\end{eqnarray*}
For some constant $K>0$, we have on $[0,T]$:
\begin{equation*}
\| \breve{v} \|_{L^{2}}^{2} + \| \breve{\ell} \|^{2} + \| \breve{r} \|^{2} \leq K,
\end{equation*}
so for some $C>0$ one has in particular
\begin{equation*}
\frac{d}{dt} \big( \| \breve{v} \|_{L^{2}}^{2} + \| \breve{\ell} \|^{2} + \| \breve{r} \|^{2} \big) 
\leq C p \Big( \| \breve{v} \|_{L^{2}}^{2} + \| \breve{\ell} \|^{2} + \| \breve{r} \|^{2} \Big)^{1/p'}.
\end{equation*}
Now the unique solution of $y' = N y^{\delta}$ and $y(0)= \varepsilon>0$ for $\delta \in (0,1)$ and $N >0$ is given by
\begin{equation*}
y(t) = \Big[ (1-\delta) Nt + \varepsilon^{1-\delta}\Big]^{\frac{1}{1-\delta}}.
\end{equation*}
Hence a comparison argument proves that 
\begin{equation*}
\| \breve{v} \|_{L^{2}}^{2} + \| \breve{\ell} \|^{2} + \| \breve{r} \|^{2} \leq (Ct)^{p}.
\end{equation*}
We conclude that $\breve{v}=0$ for $t< 1/C$ by letting $p \rightarrow +\infty$. \par
\ \par
\subsection{Proof of Theorem \ref{ThmYudo} (Existence)}
We now briefly explain how to deduce Theorem \ref{ThmYudo} from Theorem \ref{ThmYudo0}, by a quite straightforward adaptation of the methods of  \cite{lions}. \par
Consider an initial data for the vorticity $w_{0} \in L_c^{p}(\overline{{\mathcal F}_0})$ with $p \in (2,+\infty)$.
Let us introduce a sequence $(w_{0}^{n})_{n \in \N} \in (L_c^{\infty}(\overline{{\mathcal F}_0}))^{\N}$ converging to $w_{0}$ in $L^{p}({\mathcal F}_{0})$. Let us consider $(\ell_{n},r_{n},v_{n})$ the corresponding solutions given by Theorem \ref{ThmYudo0} in the body frame. \par
We use the conservation of the energy given in Proposition \ref{PropKirchoffSueur}, of the $L^p$ norm of the vorticity and of the circulation around the body an proceed as in Section \ref{Sec:APE} to deduce that for any $T>0$:
\begin{equation} \label{APEDiPernaMajda}
\| \rho_{\omega_{n}} \|_{L^{\infty}(0,T)} +  \| \ell_{n} \|_{L^{\infty}(0,T)} + \| r_{n} \|_{L^{\infty}(0,T)} + \| v_{n} \|_{L^{\infty}(0,T;W^{1,p}({\mathcal F}_{0}))} \leq C,
\end{equation}
where we used the notation \eqref{DefRhof}. (As a matter of fact, the proof could be simpler, since we have a classical flow associated to $v_{n}$.) \par
Using the definition \eqref{DefNablaMu} of $\mu$, Lemma \ref{LemProjLeray}, \eqref{valinfini} and \eqref{APEDiPernaMajda}, we deduce that
\begin{equation*}
\| \nabla \mu_{n} \|_{L^{\infty}(0,T;L^{p}({\mathcal F}_{0}))} \leq C.
\end{equation*}
Going back to \eqref{CalcPression2} and then to \eqref{CalcPression}, we deduce that
\begin{equation*}
\| \ell_{n} \|_{W^{1,\infty}(0,T)} + \| r_n \|_{W^{1,\infty}(0,T)}  + \| \nabla q_{n} \|_{L^{\infty}(0,T;L^{p}({\mathcal F}_{0}))}\leq C.
\end{equation*}
Now, using (\ref{Euler11}) and \eqref{valinfini}, we deduce that
\begin{equation*}
\| \partial_{t} v_{n} \|_{L^{\infty}(0,T;L^{p}({\mathcal F}_{0}))}\leq C.
\end{equation*}
Moreover, \eqref{Euler11} can be written as follows:
\begin{eqnarray*}
-  \partial_t \hat{v}_{n}=  \nabla q_{n} +\left(v_{n}-\ell_{n} \right) \cdot\nabla v_{n} - r_{n} x^\perp \cdot\nabla \hat{v}_{n}
+r_{n} \hat{v}_{n}^\perp + (\alpha+\gamma) r_{n} [H^\perp - (x^\perp \cdot\nabla) H] ,
  \end{eqnarray*}
where  $\hat{v}_{n}:= v_{n} - (\gamma + \alpha) H$, with $\alpha$ given by \eqref{DefAlpha}.
Using \eqref{HSeriesLaurent2}, we infer that 
\begin{equation} \label{HalInfini2}
\partial_{t} v_{n}= \partial_t \hat{v}_{n}  = {\mathcal O}\left(\frac{1}{|x|^2}\right) ,
\end{equation}
uniformly in $n$ and in $t$, so that $\partial_{t} v_{n}$ is bounded in $L^{\infty}(0,T;L^{q}({\mathcal F}_{0}))$ for any $q \in (1,p ]$. \par
Hence by a straightforward compactness argument, we extract a sequence, that we still index by $n$ such that
\begin{equation*}
\ell_{n} \cvwstar \ell, \ r_{n} \cvwstar r \ \text{ in } \ W^{1,\infty}_{loc}(\R^{+}),\ 
\nabla q_{n} \cvwstar \nabla q \ \text{ in } \ L^{\infty}_{loc}(\R^{+};L^{p}({\mathcal F}_{0})),
\ \text{ and }
v_{n} \cvwstar v \ \text{ in } \ L^{\infty}_{loc}(\R^{+};W^{1,p}({\mathcal F}_{0})),
\end{equation*}
and, using Aubin-Lions' lemma, such that
\begin{equation*}
v_{n} \longrightarrow v \text{ in } L^{\infty}_{loc}(\R^{+} \times {\mathcal F}_{0}).
\end{equation*}
This is enough to pass to the limit into the equation, in the sense of distributions. \par
Going back to the original variables, we get that
\begin{equation*}
u \in L^{\infty}(0,T;W^{1,p}({\mathcal F}(t))), \ 
\nabla p \in L^{\infty}(0,T;L^{p}({\mathcal F}(t))),
\end{equation*}
satisfy \eqref{Euler1}. Using $(u\cdot \nabla) u = \div (u \otimes u)$, we deduce that $\partial_{t} u \in L^{\infty}(0,T;L^{p}({\mathcal F}(t)))$, and, considering its behaviour at infinity,  $\partial_{t} u \in L^{\infty}(0,T;L^{q}({\mathcal F}(t)))$ for all $q \in (1,p]$. We deduce finally that $\nabla p \in L^{\infty}(0,T;L^{q}({\mathcal F}(t)))$ for all $q \in (1,p]$. \par
\subsection{Proof of Lemma \ref{support} and of Theorem \ref{ThmYudo} (Additional properties)}
\label{PreuveYudoAdd}
We begin by proving Lemma \ref{support}. This will in particular establish the assertion in Theorem \ref{ThmYudo} concerning the compactness of the support of $w(t,\cdot)$. \par
\begin{proof}[Proof of Lemma \ref{support}]
In the case $p=+\infty$, we have a unique flow $\Phi(t,s,x)$ associated to $v^{\eps}-\ell^{\eps}-r^{\eps} x^\perp$, since this vector field is log-Lipschitz uniformly in time, and the conclusion follows easily, because
\begin{equation*}
\omega(t,x) =w_{0}(\Phi(0,t,x)).
\end{equation*}
Let us discuss the main case, that is $p< +\infty$.
We will use the renormalization theory for which we refer to  \cite{DiPernaLions,lions,bouchut}. In particular it follows from \cite[Theorem 3.2]{bouchut}, that  $\omega^{\eps} $ is a renormalized solution of  \eqref{vorty1}, in the sense that for any $\beta \in \text{Lip} (\R;\R)$, 
\begin{equation*}
\partial_t  \beta (\omega^{\eps} ) + \div 
\big(  \beta (\omega^{\eps} ) (v^{\eps}-\ell^{\eps}-r^{\eps} x^\perp) \big) =0,
\end{equation*}
in $(0,T) \times \mathcal{F}^{\eps}_{0} $ in the sense of  distributions. We deduce that for any $\psi \in C^\infty_{c} ( [0,T] \times {\mathcal{F}^{\eps}_{0}} )$ ,
\begin{equation}
\label{nope}
\int_{\mathcal{F}^{\eps}_{0}} \beta( \omega^{\eps}  (t,\cdot)) \psi     (t,\cdot) 
=
\int_{\mathcal{F}^{\eps}_{0}} \beta (w_{0}) \psi     (0,\cdot) 
+ \int_{ [0,t] \times  \mathcal{F}^{\eps}_{0}} (\partial_t  \psi   +  (v^{\eps}-\ell^{\eps}-r^{\eps} x^\perp) \cdot \nabla_{x }  \psi  ) \beta( \omega^\eps)  .
\end{equation}
%
%
%
Now let us prove that \eqref{nope} holds also for any $\psi \in C^\infty_{c} ( [0,T] \times \overline{\mathcal{F}^{\eps}_{0}} )$ and any bounded $\beta \in \text{Lip} (\R;\R)$. Let $\eta \in C^{\infty}(\R^{+};\R)$ such that $\eta =1$ in $[0,1]$ and $\eta=0$ in $[2,+\infty)$.
Given $\phi \in C^\infty_{c} ( [0,T] \times \overline{\mathcal{F}^{\eps}_{0}} )$ and a bounded $\beta \in \text{Lip} (\R;\R)$, we apply \eqref{nope} to $\psi(t,x):=\phi(t,x) \eta(d(x,\partial {\mathcal S}^{\varepsilon}_{0})/{\delta})$ for $\delta>0$ small. Using $\nabla d(x,\partial {\mathcal S}^{\varepsilon}_{0}) \cdot [v^{\varepsilon}-\ell^{\varepsilon} -r^{\varepsilon} x^{\perp}]=0$ on $\partial {\mathcal S}^{\varepsilon}_{0}$, the uniform continuity of $\nabla d(x,{\mathcal S}_{0}^{\varepsilon}) \cdot [v^{\varepsilon}-\ell^{\varepsilon} -r^{\varepsilon} x^{\perp}]$ on some compact neighborhood of $[0,T] \times {\mathcal S}_{0}^{\varepsilon}$ and letting $\delta \rightarrow 0^{+}$, we deduce the claim. \par
Using Lebesgue's dominated convergence theorem, we deduce that \eqref{nope} is still valid when $\beta(t):=|t|^{p}$ and when $\psi$ is replaced by some $\phi \in C^\infty ( [0,t] \times \overline{\mathcal{F}^{\eps}_{0}} )$, bounded as well as its first derivatives. As a consequence we obtain that  for all $t$, 
\begin{equation}
\label{nope15}
\int_{\mathcal{F}^{\eps}_{0}} | \omega^{\eps}  (t,\cdot))|^{p} \phi     (t,\cdot) 
=
\int_{\mathcal{F}^{\eps}_{0}} |w_{0}|^{p} \phi     (0,\cdot) 
+ \int_{ [0,t] \times  \mathcal{F}^{\eps}_{0}} (\partial_t  \phi   +  (v^{\eps}-\ell^{\eps}-r^{\eps} x^\perp) \cdot \nabla_{x }  \phi  ) |\omega^\eps|^{p}  .
\end{equation}

Now, as in \cite{CricriLylyne}, let $\phi_{0}$ denote a smooth nondecreasing function from $\R$ to $\R$ such that  $\phi_{0} (s) =0 $ for $s \leq 1$ and   $\phi_{0} (s) = 1 $ for $s \geq 2$. We also define $\phi (s,x) := \phi_{0} ( | x | / R(s) ) $, where
\begin{equation*}
R(t) := \rho^{\varepsilon}(0) + \int_{0}^{t} \| v^{\varepsilon} - \ell^{\varepsilon} \|_{L^{\infty}(\R^{2} \setminus B(0,1))}.
\end{equation*}
With this test function \eqref{nope15} now reads
\begin{equation}
\label{nope2}
\int_{\mathcal{F}^{\eps}_{0}} | \omega^{\eps}  (t,\cdot)    |^{p}  \phi_{0} \left(\frac{|x|}{R}\right)
=
\int_{[0,t] \times \mathcal{F}^{\eps}_{0}} \left( ( v^{\eps}-\ell^{\eps}-r^{\eps} x^\perp)\cdot \frac{x}{|x|} - \frac{R'}{R} |x| \right)  \frac{\phi'_{0} \left( \frac{|x|}{R} \right) }{R}  | \omega^{\eps} |^{p}  .
\end{equation}
Since $\phi'_{0}$ has non trivial values only when $\frac{|x|}{R}  \geq 1$, we get that 
\begin{equation}
\label{nope3}
\int_{\mathcal{F}^{\eps}_{0}} | \omega^{\eps}  (t,\cdot)    |^{p}  \phi_{0} \left(\frac{|x|}{R}\right)
\leq
\int_{ [0,t] \times  \mathcal{F}^{\eps}_{0}} \left(  \| v^{\varepsilon} - \ell^{\varepsilon} \|_{L^{\infty}(\R^{2} \setminus B(0,1))} - R' \right) \frac{\phi'_{0} \left( \frac{|x|}{R} \right)}{R}  | \omega^{\eps}   |^{p} = 0 ,
\end{equation}
which proves the lemma.
\end{proof}
Let us finally establish the last properties of the solutions announced in Theorem \ref{ThmYudo}. \par
\ \par
\noindent
\begin{proof}[End of the proof of Theorem \ref{ThmYudo}]
That the quantities mentioned in the statement are preserved over time can be seen as follows.
\begin{itemize}
\item Concerning the conservation of energy, this was proven in Proposition \ref{PropKirchoffSueur}. 
\item Concerning the conservation of $\| w(t) \|_{q}$ for $q \in [1,p]$, we use again \eqref{nope}. By using Lebesgue's dominated convergence theorem, it is easy to see that one can apply it to $\psi=1$ and $\beta (t)=|t|^{q}$, which yields the result.
\item The conservation of $\int_{{\mathcal F}_{0}} w(t,x) \, dx$ can be seen likewise (this is just the trivial case $\beta(t)=t$.)
\item Finally, the conservation of $\int_{\partial {\mathcal S}_{0}} u \cdot \tau \, ds$ can be seen as follows. For $R>0$ large enough, due to Lemma \ref{support}, one can see that $v(t,\cdot)$ is harmonic (and therefore smooth) outside $B(0,R)$. It follows then from the usual Kelvin's circulation theorem that the circulation on large circles is conserved along the flow. Combined with Green's formula and the conservation of $\int_{{\mathcal F}_{0}} w(t,x) \, dx$, this proves the claim.
\end{itemize}
Finally, we can deduce the continuity in time of $v$ with values in $W^{1,p}({\mathcal F}_{0})$ (for $p < +\infty$). Indeed, we already have $v \in L^{\infty}_{loc}(\R^{+};W^{1,p}({\mathcal F}_{0}))$ and  $v \in C^{0}(\R^{+};L^{p}({\mathcal F}_{0}))$; consequently one has $v \in C^{0}(\R^{+};W^{1,p}({\mathcal F}_{0})-w)$. It follows that
\begin{equation*}
w(s,\cdot) \cv w(t,\cdot) \text{ in } L^{p}({\mathcal F}_{0})-w  \text{ as } s \rightarrow t.
\end{equation*}
Due to the conservation of $\| w(t) \|_{p}$, this convergence is strong, and the claim follows. 
\end{proof}
\section{Hamiltonian structure of the limit system}
\label{Sec:Hamilton}
First we endow the manifold $ \mathcal{P}$ of the triplets $(w,h,\xi):=(w, h, mh')$ with a Poisson structure (see \cite{ArnoldKhesin}). That is to say, we endow ${\mathcal P}$ with a bracket $\{ \cdot  , \cdot \}$ acting on $C^{\infty}$ functionals $F:{\mathcal P} \rightarrow \R$, bilinear and skew-symmetric,  satisfying the Jacobi and the Leibniz identities. Here this is obtained by setting, for any smooth functionals $F^1 , F^2$ on $ \mathcal{P}$, 
\begin{equation*}
\{F^1 , F^2 \} := \gamma   F^1_\xi  \cdot    (F^2_\xi)^\perp   -  ( F^1_\xi \cdot F^2_h -   F^1_h \cdot F^2_\xi )  - \int_{\R^2 }  w \ \nabla_{x} F^1_w \cdot  \nabla_{x}^\perp F^2_w ,
\end{equation*}
where  $F_w$, $F_h$ and $ F_\xi$ denote the derivatives with respect to $ w, h$ and $\xi$ of a functional $F$. The properties cited above are clear from this definition. \par
Moreover, $\mathcal{H}$ (given by \eqref{HamiltonienLimite}) can also be seen as a functional on the manifold $\mathcal{P}$, and this endows the system with a Hamiltonian structure in the following sense. The following computations are valid either when both $F$ and the solution are smooth, or in the framework considered in this paper when $F={\mathcal H}$.
\begin{Proposition}\label{ProPoisson}
When $( w, h,\xi)$ solves  the equations \eqref{EulerPoint}--\eqref{EPU} we have for any smooth functional $F$, the ordinary differential equation
\begin{equation}
\label{haha}
\frac{d}{dt} F  = \{F , \mathcal{H} \},
\end{equation}
where $F$ and $ \mathcal{H}$ in \eqref{haha} stand respectively for $ F ( w, h,\xi) $ and $   \mathcal{H}( w, h,\xi) $.
\end{Proposition}
\begin{proof}[Proof of Proposition \ref{ProPoisson}]
According to the chain rule, we have 
\begin{equation*}
\frac{d  }{dt} F ( w, h,\xi)=  \int_{\R^2  } F_w \frac{\partial w }{\partial t} + F_h \cdot h' + F_\xi \cdot  m h''.
\end{equation*}
Using the equations \eqref{EulerPoint}--\eqref{EPU} we  arrive to 
\begin{equation*}
\frac{d  }{dt}  F ( w, h,\xi) =  - \int_{\R^2  }  F_w \Big( \tilde{u}+ \frac{\gamma}{2\pi} \frac{(x-h(t))^{\perp}}{|x-h(t)|^{2}} \Big) \cdot\nabla_x  w  +  F_h \cdot h'+ \gamma F_\xi \cdot \Big(h'(t) -  \tilde{u}(t,h(t))\Big)^\perp  =: I_1 + I_2 +I_3 .
\end{equation*}
On the other side the derivatives of $ \mathcal{H}$ are 
\begin{equation*}
\mathcal{H}_w = -  \int_{\R^2  }   G( \cdot - y) \, w(y) \, dy - \gamma G(  \cdot - h(t)) , \quad \mathcal{H}_h =  \gamma \tilde{u}(t,h(t))^\perp , \quad \mathcal{H}_\xi =  \frac{\xi}{m} .
\end{equation*}
Then, integrating by parts, we find 
\begin{equation*}
I_1 =   \int_{\R^2} w \nabla^\perp \left(  \int_{\R^2  }   G( \cdot - y ) \, w(y) \, dy + \gamma G( \cdot - h(t) ) \right) \cdot \nabla F_w 
= -  \int_{\R^2  } w  \nabla F_w  \cdot \nabla^\perp  \mathcal{H}_w .
\end{equation*}
On the other side, we have
\begin{gather*}
I_2 = h' \cdot F_h =  \mathcal{H}_\xi \cdot F_h, \\
I_3 =  \frac{ \gamma}{m} \xi^\perp  \cdot  F_\xi - \gamma   \tilde{u}(t,h(t))^\perp \cdot F_\xi 
= \gamma    \mathcal{H}_\xi^\perp  \cdot  F_\xi -  \mathcal{H}_h  \cdot  F_\xi  ,
\end{gather*}
which yields  \eqref{haha}.
\end{proof}
As a simple corollary of \eqref{haha} and of the skew-symmetry of the bracket we get that $\mathcal{H}$ is conserved by the solutions of  \eqref{EulerPoint}--\eqref{EP0}. \par
\ \par
\noindent
{\bf Acknowledgements.} The first and third authors were partially supported by the Agence Nationale de la Recherche, Project CISIFS,  grant ANR-09-BLAN-0213-02. The second author is partially supported by the Agence Nationale de la Recherche, Project MathOc\'ean, grant ANR-08-BLAN-0301-01.

\end{document}